\documentclass[12pt]{amsart}
\usepackage{etex}
\usepackage{etoolbox}
\patchcmd{\thebibliography}{*}{}{}{}
\usepackage{amssymb}
\usepackage{cite}
\usepackage{booktabs}
\usepackage{url}
\usepackage{hyphenat}
\usepackage{mathtools}
\usepackage{pifont}
\usepackage[all,cmtip]{xy}
\usepackage{ifpdf}
\usepackage{enumitem}
\usepackage{tikz-cd}
\usepackage{xcolor}
\usepackage{graphicx}
\usepackage{subcaption}
\usepackage[lmargin=1.25in,rmargin=1.25in,tmargin=1in,bmargin=1in]{geometry}
\usepackage{mathrsfs}
\usepackage{overpic}
\usepackage{xr-hyper}

\definecolor{darkblue}{rgb}{0,0,0.4} 
\usepackage[colorlinks=true, citecolor=darkblue, filecolor=darkblue, linkcolor=darkblue,urlcolor=darkblue]{hyperref} 
\usepackage[all]{hypcap}

\usetikzlibrary{matrix,arrows,3d}
\tikzset{zxplane/.style={canvas is zx plane at y=#1,very thin}}
\tikzset{xyplane/.style={canvas is xy plane at z=#1,very thin}}
\tikzset{yzplane/.style={canvas is yz plane at x=#1,very thin}}

\newcommand{\ZZ}{\mathbb Z}
\newcommand{\QQ}{\mathbb Q}

\newcommand{\FF}{\mathbb F}

\newcommand{\isom}{\simeq}

\newcommand{\tensor}{\otimes}

\newcommand{\lbracket}{[}
\newcommand{\rbracket}{]}

\DeclareMathOperator{\Sym}{Sym}

\DeclareMathOperator{\Hom}{Hom}

\DeclareMathOperator{\id}{id}

\DeclareMathOperator{\Cone}{Cone}
\DeclareMathOperator{\Conf}{Conf}

\theoremstyle{plain}

\newtheorem{theorem}{Theorem}[section]
\newtheorem{proposition}[theorem]{Proposition}
\newtheorem{lemma}[theorem]{Lemma}

\newtheorem{definition}[theorem]{Definition}

\newtheorem{condition}[theorem]{Condition}

\theoremstyle{definition}

\theoremstyle{remark}
\newtheorem{example}[equation]{Example}
\newtheorem{remark}[equation]{Remark}

\newcommand{\HF}{\mathit{HF}}

\newcommand{\HFh}{\widehat{\mathit{HF}}}

\newcommand{\CF}{{\mathit{CF}}}
\newcommand{\CFh}{\widehat{\mathit{CF}}}

\newcommand{\x}{\mathbf x}

\newcommand\HH{\mathit{HH}}

\newcommand\Hochschild\HH

\newcommand{\balpha}{{\boldsymbol{\alpha}}}
\newcommand{\bbeta}{{\boldsymbol{\beta}}}

\newcommand{\tbeta}{{\widetilde{\beta}}}
\newcommand{\tgamma}{{\widetilde{\gamma}}}

\newcommand{\pralpha}{\alpha_{\textrm{pre}}}
\newcommand{\prbeta}{\beta_\textrm{pre}}
\newcommand{\bbbeta}{\bold{\eta}}
\newcommand{\bteta}{\widetilde{\bold{\eta}}}

\newcommand{\twibeta}{\underline{\beta\beta'\beta''}}
\newcommand{\twitbeta}{\underline{\tilde{\beta}\tilde{\beta'}\tilde{\beta''}}}
\newcommand{\ttheta}{\tilde{\theta}}
\newcommand{\twigamma}{\underline{\gamma\gamma'\gamma''}}
\newcommand{\twitgamma}{\underline{\tilde{\gamma}\tilde{\gamma'}\tilde{\gamma''}}}

\newcommand{\HFheq}{\HFh_{\ZZ/2\ZZ}}

\newcommand{\mA}{\mathcal{A}}

\newcommand{\Field}{{\FF_2}}

\newcommand{\Heegaard}{\mathcal{H}}
\newcommand{\HD}{\Heegaard}

\DeclareMathOperator{\Fix}{Fix}

\newcommand{\Index}{\{0,1,\infty\}^k}

\newcommand{\Z}{\ZZ/2\ZZ}

\makeatletter
\newcommand\honestalg[3]{\bigl\lbracket
\begin{smallmatrix} #1\@ifempty{#3}{}{&#3} \\ #2 \end{smallmatrix}
\bigr\rbracket}

\makeatother

\newcommand{\fix}{\mathit{fix}}

\makeatletter

\newcommand*\wthelper[2]{%
        \hbox{\dimen@\accentfontxheight#1%
                \accentfontxheight#11.3\dimen@
                $\m@th#1\widetilde{#2}$%
                \accentfontxheight#1\dimen@
        }%
}
\newcommand*\accentfontxheight[1]{%
        \fontdimen5\ifx#1\displaystyle
                \textfont
        \else\ifx#1\textstyle
                \textfont
        \else\ifx#1\scriptstyle
                \scriptfont
        \else
                \scriptscriptfont
        \fi\fi\fi3
}
\makeatother

\begin{document}
	\title[equivariant Floer homology of double branched covers]{Exact triangle and Spectral sequence of the equivariant Floer homology of double branched covers}
	
	\author{Jinzhou Huang}
	\email{jhuang10@caltech.edu}

	\begin{abstract}
		In \cite{OS05} the authors proved an exact triangle for the Heegaard Floer homology of the double branched covers associated to a skein sequence of links,  and constructed a spectral sequence from the reduced Khovanov homology to the Heegaard Floer homology of the double branched cover of a link. For $L\subset S^3$ a link and $\Sigma(L)$ its double branched cover, the invariant $\HFh(\Sigma(L))$ doesn't capture the information of the natural involution on $\Sigma(L)$. In this paper we study the equivariant version of $\HFh(\Sigma(L))$ and prove that similar exact triangle and spectral sequence hold for this equivariant version.
		\end{abstract}
		
		\maketitle

\section{Introduction}
		\begin{definition}\label{borel}
			Given a finite dimensional chain complex $(E,\partial)$(possibly ungraded) over $\Field$ with a $\ZZ/2\ZZ$-action $\tau$, denote its dual over $\Field$ by $E^\vee$, define the projective resolution associated to $(E,\tau)$ to be 
			\[P(E)=E^\vee\tensor\Field[[q]],\]
			with differential $\partial_{\textrm{Borel}}=\partial^\vee\tensor\id+(1+\tau)\tensor q.$
			We define the equivariant cohomology associated to $E,\partial,\tau$ to be 
			\[H_{\ZZ/2\ZZ}(E,\partial,\tau)=H^*(P(E),\partial_{\textrm{Borel}}).\]
			When the context is clear we will denote the equivariant cohomology simply by $H_{\ZZ/2\ZZ}(E)$.
			 
		\end{definition}

	\
		
		Given a symplectic manifold $(M,\omega)$, a symplectic involution $\tau:M\to M$ and a pair of Lagrangians $L_0, L_1\subset M$ each preserved by $\tau$  satisfying certain geometric conditions. Except the routine ones used to define Floer cohomology, the most important condition is the "stable normal trivialization" \cite[Definition 18]{SS}, which guarantee that the equivariant transversality can be achieved. Under these geometric conditions they construct a Floer complex $\CFh(L_0,L_1)$ with a $\ZZ/2\ZZ$ action, and from it an equivariant Floer cohomology $\HF_{\ZZ/2\ZZ}(L_0,L_1)$ and a spectral sequence $\HF^*(L_0,L_1)\tensor \Field[[q]]\Rightarrow\HF_{\Z}(L_0,L_1))$. The authors in \cite{SS} used this equivariant cohomology to prove analogous result of Borel equivariant cohomology in symplectic geometry, like a version of Smith inequality \cite[Theorem 1]{SS}. In \cite[Proposition 3.25]{HLS16} the authors proved that the equivariant Floer homology $\HF_{\Z}(L_0,L_1)$ defined above is indeed an invariant under $\Z$-equivariant Hamiltonian isotopy of $L_0$ and $L_1$.
		
		The authors in \cite{HLS16} construct the equivariant Floer homology for data $(M,\omega,L_0,L_1,\tau)$ as above satisfying some looser conditions (\cite[Hypothesis 3.2]{HLS16}). In particular they do not require any conditions related to equivariant transversality(see \cite[Section 3]{HLS16}). The general definition, however, is quite complicated and not easy to compute. But as indicated above, when we can achieve transversality the equivariant Floer homology defined in \cite{SS} and \cite{HLS16} coincides \cite[Proposition 4.3]{HLS16}, which gives us an easier way to describe it. 
		
		In the context of Heegaard Floer homology, given a pointed link $(L,p)$ in $S^3$, the authors in \cite{HLS16} defined an equivariant Heegaard Floer homology $\HFh_{\Z}(\Sigma(L),p)$, which can be computed in the following procedures. We choose a bridge diagram of $L$ on the plane. By lifting the bridge diagram to the branched double along the ends points of the bridges we get a Heegaard diagram $(\Sigma, \alpha, \beta, z)$ associated to $\Sigma(L)$(for more details see Section \ref{diagram}). Due to the specific choice of Heegaard diagram we made, for generic choices of equivariant almost complex structure $J$, we can achieve transversality for all the moduli spaces of holomorphic disks. Thus, we get a chain complex $\CFh(\Sigma(L),p)=\CFh(\Sigma,\alpha,\beta,z;J)$ with an obvious $\ZZ/2\ZZ$ action $\tau$. The equivariant Heegaard Floer homology $\HFh_{\Z}(\Sigma(L),p)$ can be computed as $H_{\Z}(\CFh(\Sigma(L),p), \tau)$(see Definition \ref{borel}).
		
		While the authors in \cite{HLS16} only proved that there is a spectral sequence 
		\[\HFh(\Sigma(L),p)\tensor\Field[q,q^{-1}]\Rightarrow(\Field\oplus\Field)^{\tensor(|L|-1)}\tensor\Field[q,q^{-1}]\] which is an invariant of the pointed link $(L,p)$,  in \cite[Section 6]{KANG18} the author proved that the equivariant Floer homology associated to a pointed knot $(K,p)$ in $S^3$ satisfies naturality and functoriality.
		
		In \cite{HLS16} the authors works in $\Field[q]$ coefficients(so the equivariant Floer homology is calculated by $P'(E)=E^\vee\tensor\Field[q]$, compare Definition \ref{borel}). To simplify our proof we will work in the $\Field[[q]]$ coefficient in our article, so we can use Lemma \ref{alg}. 
		
		Our main theorem is the following.

		\begin{theorem}\label{main}
			Let $(L,p)$ be a link in $S^3$. For unoriented skein moves as in Figure (\ref{skein}), if $p$ is away from a three-ball where the $0$, $1$-resolutions happen(so $p$ can be seen as lying on $L$, $L_0$ and $L_1$), then there exists a chain map $f:\CFh(\Sigma(L),p)\to \CFh(\Sigma(L_0),p)$ and a $\Field[\ZZ/2\ZZ]$-quasi isomorphism
			\[\Cone(f:\CFh(\Sigma(L),p)\to\CFh(\Sigma(L_0),p))\to \CFh(\Sigma(L_1),p)\].
		\end{theorem}

		\begin{theorem}\label{triangle}
			There is an exact sequence
			\begin{equation}\label{seq}
			\to\HFh_{\ZZ/2\ZZ}(\Sigma(L),p)\to \HFh_{\ZZ/2\ZZ}(\Sigma(L_1),p)\to \HFh_{\ZZ/2\ZZ}(\Sigma(L_0),p)\to\HFh_{\ZZ/2\ZZ}(\Sigma(L),p)\to.
			\end{equation}
		\end{theorem}
		
		\

		\begin{remark}
			If we assume $|L_0|=|L_1|=|L|+1$, after doing tensor product with $\Field[q^{-1},q]]$, then \cite[Lemma 6.11]{HLS16} gives us an explicit description of the maps $\HFh_{\ZZ/2\ZZ}(\Sigma(L),p)\to \HFh_{\ZZ/2\ZZ}(\Sigma(L_1),p)$ and $\HFh_{\ZZ/2\ZZ}(\Sigma(L_0),p)\to\HFh_{\ZZ/2\ZZ}(\Sigma(L),p)$. Following the same argument in \cite[Lemma 6.11]{HLS16} we can also show that the map $\HFh_{\ZZ/2\ZZ}(\Sigma(L_1),p)\to \HFh_{\ZZ/2\ZZ}(\Sigma(L_0),p)$ is trivial after doing tensor product with $\Field[q^{-1},q]]$.
		\end{remark}
		
		\
		
		Given any positive integer $N$, we can also consider the equivariant Heegaard Floer homology with coefficient $\Field[q]/q^N$, which we will denoted by $\HFheq(\Sigma(L),p;\Field[q]/q^N)$. It is computed by the chain complex $P(\CFh(\Sigma(L),p))\tensor \Field[[q]]/q^N$(see Definition \ref{borel}). By modeling on the constructions in \cite[Section 6]{OS05}, we get the following:
		\begin{theorem}\label{spectral}
			Let $L\subset S^3$ be a link. For each $N\ge1$, there is a spectral sequence whose $E^2$ page is isomorphic to $\widetilde{Kh}(L,p)\tensor \Field[q]/q^N$, the reduced Khovanov homology, and converges to $\HFheq(\Sigma(L),p;\Field[q]/q^N)$.
		\end{theorem}
		To prove Theorem \ref{triangle} we need an algebraic lemma.
		
			\begin{lemma}\label{alg}
			Given two finite dimensional chain complexes $E$, $E'$ over $\Field$ with $\ZZ/2\ZZ$ actions $\tau$ and $\tau'$ respectively, if $f$ is a  $\ZZ/2\ZZ$-equivariant quasi-isomorphism between $E$ and $E'$, then $f$ induces a quasi-isomorphism between $P(E)$ and $P(E')$.
		\end{lemma}
		
		{\bf{Proof}:}
		Consider the filtration on $P(E)$(and $P(E')$) given by the $q$-action, it induces a spectral sequence whose first page is $H^*(E^\vee)\tensor\Field[[q]]$(and $H^*(E'^\vee)\tensor\Field[[q]]$) and the infinity page is $H^*_{\ZZ/2\ZZ}(E)$(and $H^*_{\ZZ/2\ZZ}(E')$). $f$ induces an isomorphism between the first pages, thus also induces an isomorphism on the infinity page by Eilenberg-Moore Comparison Theorem (see \cite[Theorem 5.5.11]{Wei}).
		\qed
		
		\

		{\bf{Proof of Theorem \ref{triangle} with Theorem \ref{main}:}}

		For convenience we omit the basepoint $p$ write the chain complexes $\CFh(\Sigma(L))$, $\CFh(\Sigma(L_0))$, $\CFh(\Sigma(L_1))$ as $E$, $E_0$ and $E_1$, and omit the notations of their $\ZZ/2\ZZ$-actions for simplicity.
		
		By Theorem \ref{main} there is a $\ZZ/2\ZZ$-equivariant chain map $f:E\to E_0$ such that there exists a $\ZZ/2\ZZ$-equivariant quasi-isomorphism $g$ between $\Cone(E\to E_0)$ and $E_1$. Moreover, there is a natural identification between the chain complexes $P((\Cone(f:E\to E_0)))$ and $\Cone(f^*\tensor \Field[[q]]:P(E_0)\to P(E))$, thus by Theorem \ref{main} and Lemma \ref{alg} there is a quasi-isomorphism between $P(E_1)$ and $P(\Cone(f^*\tensor \Field[[q]]:P(E_0)\to P(E))$, and this proves Theorem \ref{main}.
		\qed  
		
		\

		{\bf{Acknowledgements.}} The author would like to thank his advisor Yi Ni for proposing this problem and the interesting discussions in the course of this work. The author would like to thank Gheehyun Nahm for explaining his work \cite{Nah25a} and \cite{Nah15b} to him.

\

		{\bf{Disclosure of computational assistance.}} The author used ChatGPT for proof exploration and
for assistance with typesetting and language clean-up. The author is responsible for verifying the
arguments and citations and for the contents of the paper.

		\section[diagram]{Heegaard diagrams of branched double covers and equivariant Heegaard Floer of the branched double cover}\label{diagram}

\subsection{Definition of equivariant Heegaard Floer homology of branched double cover via bridge diagrams}\label{1.1}

\

\begin{definition}\cite[Section 3]{KANG18}\label{bridge}
	
	\

A based link is a pair $(L,p)$ where $L$ is a link in $S^3$ and
$p\in L$ is a choice of a basepoint. Given a genus $0$ Heegaard surface
$S\subset S^3$, a based link $(L,p)$ is in a \emph{bridge position} with
respect to $\Sigma$ if, for a Heegaard splitting
\[
S^3=H_a\cup_\Sigma H_b,
\]
the connected arcs $\{a_i\}$ and $\{b_i\}$, $1\le i\le g+1$, given by
\[
\bigcup a_i=L\cap H_a,
\qquad
\bigcup b_i=L\cap H_b,
\]
satisfy the following conditions.
\begin{itemize}
	\item There exist disks $D_{a_i}$ and $D_{b_j}$ such that
	\[
	a_i\subset \partial D_{a_i}\subset a_i\cup\Sigma
	\qquad\text{and}\qquad
	b_j\subset \partial D_{b_j}\subset b_j\cup\Sigma.
	\]
	
	\item The disks $D_{a_i}$ and $D_{b_j}$ can be chosen to have
	pairwise disjoint interiors.
	
	\item $p\in L\cap\Sigma$.
\end{itemize}

If
\[
\partial D_{a_i}=a_i\cup A_i
\qquad\text{and}\qquad
\partial D_{b_j}=b_j\cup B_j,
\]
where $A_i$ and $B_j$ are simple arcs on $S$, we say that
$(\{A_i\},\{B_j\})$ is the \emph{bridge diagram} for the based link $(L,p)$.
Given a bridge diagram, by taking the branched double cover of the
whole diagram, with the branching locus given by $L\cap\Sigma$, and
removing the curves which contain the basepoint $p$, gives a Heegaard
diagram
\[
\HD=({\Sigma},\balpha,\bbeta,p)
\]
together with the covering $\mathbb{Z}_2$-action, where the
alpha(beta)-curves are given by the inverse images of the arcs
$A_i(B_i)$.

\end{definition}

The author considers the bridge diagram embedded in $S^3$ for naturality reason, however in our context we may ignore the embedding and focus on the diagram. Note that the Heegaard diagram induced by the bridge diagram depends on basepoint $p$.

Given a bridge diagram of a based link $(L,p)$ in $S^3$ as above,  as well as the induced Heegaard diagram $\HD=(\Sigma,\balpha,\bbeta, p)$, for a generic choice of family of equivariant complex structures $\{J_s\}_{0\le s\le 1}$ satisfying certain conditions necessary to define Heegaard Floer homology(see Definition \ref{complex}), $\{J_s\}_{0\le s\le 1}$ achieves transversality for all homotopy classes of Whitney disks $\phi$ in $(\Sym^g(\Sigma\backslash\{z\}),T_\alpha, T_\beta, z)$ with Maslov index less than or equal to $1$ and $n_z(\phi)=0$ in the sense of Definition \ref{mapcondition}(see Proposition \ref{transversality}). Thus, the differential $\partial_{J_s}$, defined by counting $J$-holomorphic disks with respect to the family of equivariant almost complex structures  $J_s$ as above gives a differential on $\CFh(T_\alpha,T_\beta)$. Since $J_s$ is $\tau$-invariant, $\partial_{J_s}$ is equivariant with respect to the chain map $\tau_\#:\CFh(T_\alpha,T_\beta)\to\CFh(T_\alpha,T_\beta)$ induced by $\tau$. Now the equivariant Heegaard Floer homology associated to $(L,p)$ can be defined as $\HFheq(\Sigma(L))=\HFheq(\CFh(T_\alpha,T_\beta),\partial_{J_s},\tau_*)$(see Definition \ref{borel}).

\subsection{Preparations for the proof of exact triangle}\label{1.2}

\begin{figure}
	\centering
	\includegraphics[width=0.8\textwidth]{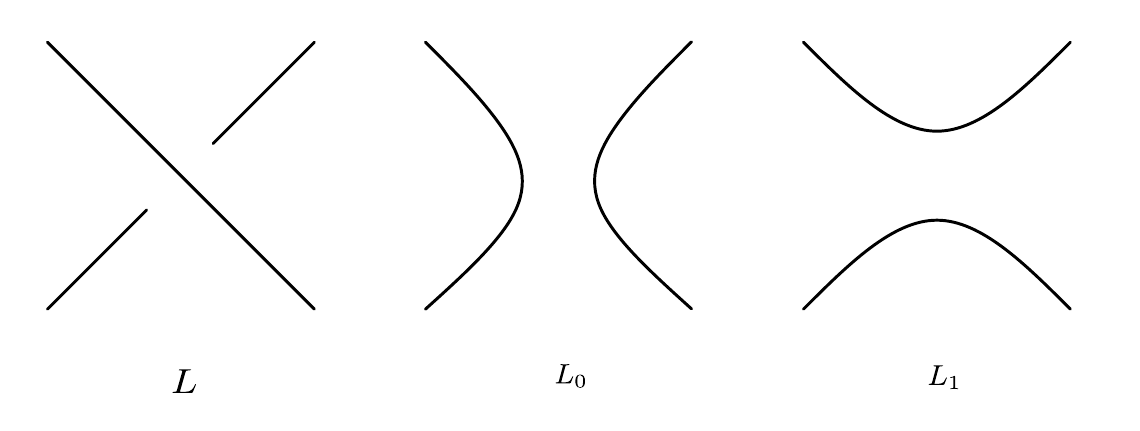}
\caption{}\label{skein}
\end{figure}

\begin{figure}
	\centering
	\begin{subfigure}{0.4\textwidth}
		\centering
\includegraphics[width=\textwidth]{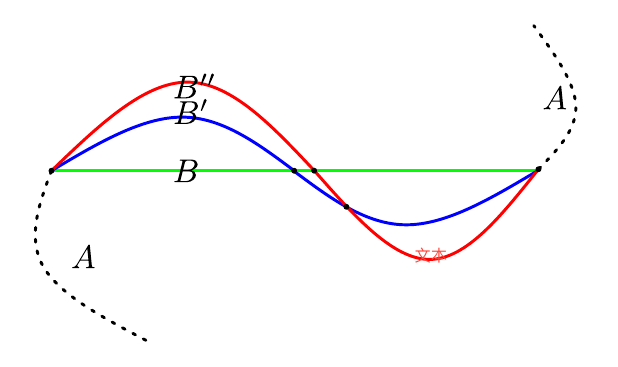}
		\end{subfigure}
		\begin{subfigure}{0.4\textwidth}
			\centering
			   \includegraphics[width=\textwidth]{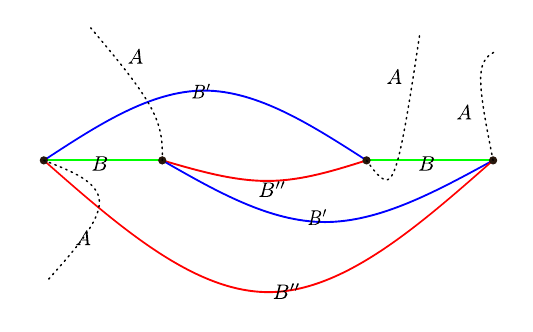}
		\end{subfigure}
			\caption{Skein triangle via bridges. Left: The bridges $B_i,B'_i,B''_i$ for $1\le i \le g-2$. Right: The bridges $B_i,B'_i,B''_i$ for $i=g-1,g$. The $B$-bridges are colored with green, the $B'$-bridges are colored with blue, and the $B''$-bridges are colored with red. The $A$-bridges are in dashed lines.}
			\label{skeindiagram}
\end{figure}

If $L_0$, $L_1$  are the $0,1$-resolutions of $L$ as in Figure (\ref{skein}), and $p$ a basepoint of $L$, $L_0$ and $L_1$ away from the region where the resolution happens, then we can construct bridge diagrams for the pointed links $(L,p)$, $(L_0,p)$ and $(L_1,p)$ as in Definition \ref{bridge} such that:
\begin{enumerate}
	\item[1)]The $A$-arcs for $L$, $L_0$ and $L_1$ are the same, denoted by $A_i$(1$\le i \le g+1$).
	
	\

	\item[2)]If we denote the $B$-arcs of $L$, $L_0$ and $L_1$ by $B_i$, $B_i'$ and $B_i''$ respectively, then there are disjoint disks $D_1\hdots D_{g-1}$ such that for each $1\le i \le g-2$, $B_i$, $B_i'$ and $B_i''$ are contained in $D_i$, and $B_i$, $B_i'$ and $B_i''$ are in the positions indicated by the left of Figure (\ref{skeindiagram}). $B_{g-1}$, $B_{g-1}'$, $B_{g-1}''$, $B_{g}$, $B_{g}'$, $B_{g}''$ are contained in $D_{g-1}$, and are in the position as indicated in the right of Figure (\ref{skeindiagram}). $A_{g+1}$, $B_{g+1}=B'_{g+1}=B''_{g+1}$ are the curves which do not contribute to the Heegaard diagram, and they are contained in $D_{g}$, and they incident at $p$.
	
	\
	
	\item[3)]Any two of the arcs $\{A_i, B_i, B_i', B_i'' (1\le i \le g), A_{g+1}, B_{g+1}=B'_{g+1}=B''_{g+1}\}$ intersect each other transversely except at their endpoints, and any three of them can only have common point at their endpoints. Moreover, at such an incident point $q$, the $A, B, B', B''$-arcs are permuted in a counterclockwise order(as in Figure (\ref{skeindiagram})).
\end{enumerate} 
\

In the following we show that the Heegaard triple $(\Sigma, \beta, \beta, \beta'',z)$ will split as connected sums of $(\Sigma_i,\beta_i,\beta'_i,\beta''_i)$ for $1\le i\le g-2$ and $(\Sigma_{g-1},\{\beta_{g-1},\beta_{g}\},\{\beta'_{g-1},\beta'_{g}\},\{\beta''_{g-1},\beta''_{g}\})$ where $\Sigma_i$ are tori for $1\le i \le g-2$ and $\Sigma_{g-1}$ is a genus $2$ oriented surface without boundary, and the connected is made at the basepoint $z$.

Here we make use of the deleted arcs $B_{g+1}=B'_{g+1}=B''_{g+1}$. For $1\le i\le g-2$ we pick disjoint embedded bands $E_i$ in $S^2\backslash(D_1\cup\hdots D_{g-2})$ connecting the boundary of $D_i$, and intersects the interior of the arc $B_{g+1}=B'_{g+1}=B''_{g+1}$ in a subarc, as shown in Figure (\ref{bands}). For each $1\le i\le g-1$ we denote $\Sigma_i$ to be the union of the preimage of $D_i$ and one connected component of the preimage of $E_i$ under the branched double cover $\Sigma\to S^2$. Then each $\Sigma_i$ is a once-punctured torus for $1\le i\le g-2$, while $\Sigma_{g-1}$ is a once-punctured genus $2$ closed surface, and the complement of the union of $\Sigma_i$ for $1\le i \le g-1$ is a $(g-1)$ times-punctured sphere containing the basepoint $z$. 

\begin{figure}
	\begin{subfigure}{0.4\textwidth}
		\includegraphics[width=\textwidth]{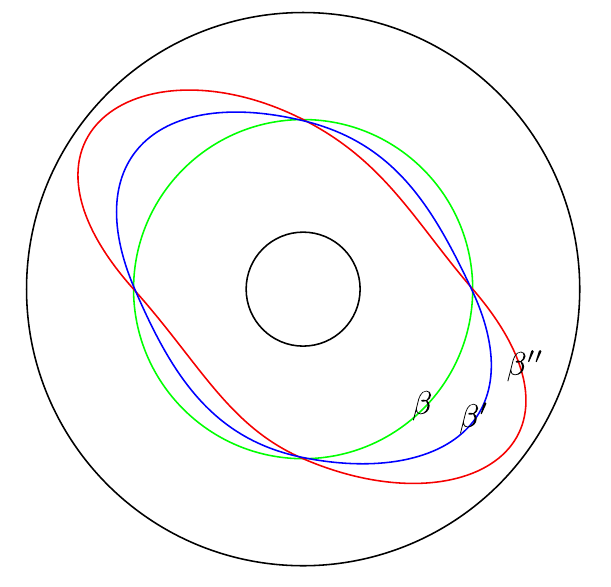}
	\end{subfigure}
	\begin{subfigure}{0.4\textwidth}
	\includegraphics[width=\textwidth]{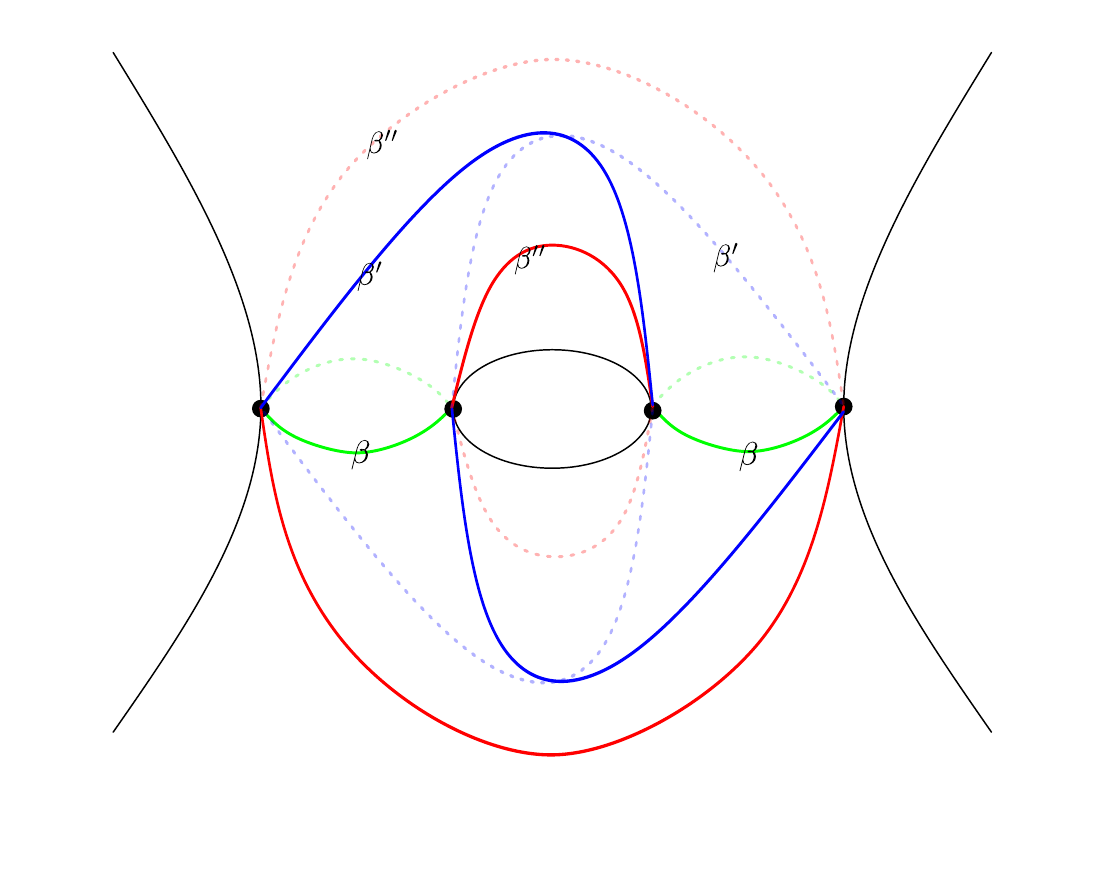}
	\end{subfigure}
	\caption{Left: The branched double cover of left of Figure (\ref{skeindiagram}). Right: The branched double cover of right of Figure (\ref{skeindiagram})}\label{double}
\end{figure}

\begin{figure}
	\includegraphics[width=0.4\textwidth]{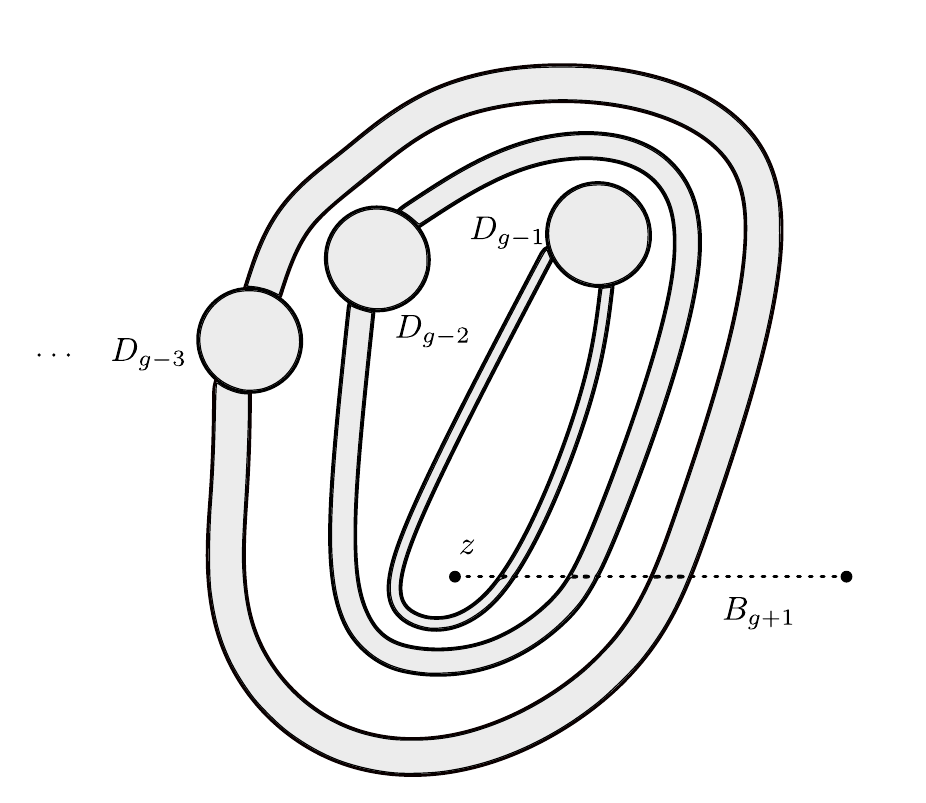}
	\caption{}
	\label{bands}
\end{figure}
To get finite count of holomorphic maps we need admissible conditions on the Heegaard diagrams, in the context of branched double cover of the bridge diagram the weakly admissible condition is automatically achieved, see \cite[Lemma 4.3, 4.3, 4.4]{KANG18}.

\begin{definition}
	By a \emph{Heegaard pre-datum} we mean a tuple $(\Sigma, \bbbeta_1,\hdots, \bbbeta_m,z)$ where $\Sigma$ is an oriented surface of some genus $g$ without boundary, and for each $1\le i\le m$, $\bbbeta_i$ consists of $g$ disjoint simple closed curves on $\Sigma$ whose homology classes are linearly independent in $H_1(\Sigma)$. We require every two curves in $\bbbeta_1\cup\bbbeta_2\hdots\cup\bbbeta_m$ intersect each other transversely, and $z\in \Sigma\backslash(\bbbeta_1\cup\bbbeta_2\hdots\cup\bbbeta_m)$. If we further require that any three different curves in $\bbbeta_1\cup\bbbeta_2\hdots\cup\bbbeta_m$ have no common point, we say that $(\Sigma, \bbbeta_1,\hdots, \bbbeta_m,z)$ is a Heegaard datum.
\end{definition}

\begin{remark}
	Sometimes we don't distinguish between the set $\bbbeta_i$ of $g$ curves on $\Sigma$ and the set $\cup_{\gamma\in\bbbeta_i}\gamma\subset \Sigma$.
\end{remark}

\begin{definition}
	A Heegaard datum is called weakly admissible if every non-trivial periodic domain has both positive and negative coefficients.
\end{definition}

\begin{definition}\label{inv}
	A Heegaard pre-datum $(\Sigma,\{\bbbeta_i\}_{1\le i\le m},z)$ is said to be \emph{pre-involutive} with respect to an orientation-preserving involution $\tau$ if $\tau$ fixes $z$ and every component of $\bbbeta_i$ curves setwise, and $\tau$ restricted to every component of $\bbbeta_i$ is orientation-reversing with two distinct fixed points, and if three curves in $\cup_{i=1}^m \bbbeta_i$ have a common point, then the common point must be a fixed point of $\tau$.
	A Heegaard datum $(\Sigma,\{\bteta_i\}_{1\le i\le m},z)$ is said to be \emph{involutive} with respect to $\tau$ if it is a small perturbation of a
	pre-involutive Heegaard datum $(\Sigma,\{\bbbeta_i\}_{1\le i\le m},z)$ with respect to $\tau$ so that every three curves in $\{\bteta_i\}_{1\le i\le m}$ have no common intersection.
\end{definition}

\begin{proposition}\label{admissible}
	Every involutive Heegaard datum is weakly-admissible.
\end{proposition}

{\bf{Proof:}}

In \cite[Lemma 4.2, 4.3, 4.4]{KANG18} the author only proved this proposition for $m=2,3,4$. We prove here the general case.

Given a pre-involutive Heegaard datum $(\Sigma,\{\bbbeta_i\}_{1\le i\le m},z)$ and its small perturbation $(\Sigma,\{\bbbeta'_i\}_{1\le i\le m},z)$ near the fixed point sets, we call the connected components of  $\Sigma\backslash\cup^m_{i=1}\bbbeta_i$ supported in the perturbation areas small domains, while the other components large domains. Given a periodic domain $D$ of $(\Sigma,\{\bbbeta_i'\}_{1\le i\le m},z)$, we will define $\tau_*(D)$ to be the periodic domain whose coefficients on the large domains are induced from $D$ via the obvious action of $\tau$ on the large domains. Obviously $\tau_*(D)$ is unique if exists.

For the existence, by \cite[Lemma 2.17]{OS04a}, we represent $D$ by a map $\Phi:F\to \Sigma$, where $F$ is a oriented  surface with boundary, $\Phi$ is not necessarily orientation-preserving, and $\Phi$ maps each boundary components diffeomorphically to some $\bbbeta_i'$ curve. Consider the Heegaard datum $(\Sigma,\{\bbbeta''_i\}_{1\le i\le m})$, where $\bbbeta''_i=\tau(\bbbeta'_i)$. For each $1\le i\le m$, there is an ambient isotopy $\{\varphi^i_t\}$ supported in a neighborhood of $\bbbeta_i$ taking $\bbbeta''_i$ to $\bbbeta'_i$. Consider a surface
\[F'=(\bigcup_{\mathclap{\substack{C\subset F\\ \textrm{is a boundary component}}}}C\times [0,1]\bigcup F)/C\times\{0\}\sim C\]
as well as a map $\Phi':F'\to \Sigma$ such that on $F\subset F'$, $\Phi'$ is given by $\tau\circ\Phi$, while on the $C\times[0,1]$ components $F'$ is given by $\varphi^i_{t}\circ\tau\circ\Phi|_C$ if $\Phi(C)\subset\bbbeta_i$. The map $\Phi'$ gives the domain $\tau_*(D)$ we desired.

Now if $D$ is a non-zero periodic domain which is non-negative, consider $D'=D+\tau_*(D)$, whose coefficient on the large domains are non-negative and invariant under $\tau$. Near a fixed point, denote the large domains surrounding the perturbation area by $D_1, D_2, \hdots, D_{2n}$ in the counterclockwise order, so that $D_i, D_{i+1}$ have common boundary in $\xi_{k_i}\in\bbbeta_{k_i}$(the index is understood modulo $n$). Denote the coefficients of $D'$ on $D_1, D_2, \hdots, D_{2n}$ by $a_1, a_2,\hdots, a_{2n}$, then by the $\tau$-invariance $a_i=a_{i+n}$, consider the multiplicity on $\xi_{k_i}$ we get $a_i-a_{i+1}=a_{i+n+1}-a_{i+n}$, thus $a_i=a_{i+1}$, that is the multiplication of $\partial D$ over $\xi_{k_i}$ is $0$ for $1\le i\le m$. Now by Definition \ref{inv} every component of every $\bbbeta_{i}$ curves has fixed points, so running the arguments above we know the multiplication of $\partial D$ is the zero $1$-chain. Now by the connectivity of the complement of small regions, all coefficients of the large domain of $D'$ are the same, thus $0$, and this implies that $D$ is a trivial domain, a contraction.
\qed
		\section[equivariant transversality and perturbations]{equivariant transversality}

In this section we will denote $(\Sigma,\{\bbbeta_i\}_{1\le i\le m},z)$ a pre-involutive Heegaard datum with respect to an involution $\tau$, where $\Sigma$ is an oriented surface of genus $g$. Denote $\bold{p}$ the set of fixed points of $\tau$. Note that $z\in \bold{p}$ as required in Definition \ref{inv}. Let $(\Sigma,\{\bbbeta'_i\}_{1\le i\le m},z)$ be an involutive Heegaard datum which is a small perturbation of the former. We require the following local condition at $\bold{p}$:

\begin{condition}\label{condition}
	At each fixed point $q\in \bold{p}$, those $\bbbeta_1, \hdots, \bbbeta_m$ which contain $q$ are arranged in counterclockwise order.
\end{condition}

First we investigate the fixed subset of $\Sym^{g}(\Sigma)$ under the induced $\ZZ/2\ZZ$-action $\tau$. Following \cite{HLS16} we consider the subset of $\Fix(\tau)$:
\[D=\{S\subset \bold{p}| |S|=g\}\subset \Fix(\tau)\subset\Sym^g(\Sigma),\]we have the following lemma:
\begin{lemma}\cite[Lemma 6.1]{HLS16}\label{discrete}
	\
	
	$D$ is discrete in $\Fix(\tau)$,
	and for each $1\le i \le m$, the $\tau$-fixed parts of $T_{\bbbeta_i}$ satisfies
	$T_{\bbbeta_i}^\fix\subset D$. In particular,
    $T_{\bbbeta_i}^\fix$ are finite sets of
	points.
\end{lemma}

Let $D_n$ denote a disk with $n$ labeled punctures on its boundary (a polygon). Label the arcs in $\partial D_n$ as $e_1,...,e_n$ in counterclockwise order, and let $p_{i,i+1}$ denote the puncture between $e_i$ and $e_{i+1}$. Let $\Conf(D_n)$ denote the moduli space of (positively oriented) complex structures on $D_n$, up to biholomorphism respecting the labeling of the edges.(For $n\ge 3$, $\Conf(D_n)$ is an $(n-3)$-dimensional ball.) The space $\Conf(D_n)$ has a natural Deligne-Mumford compactification $\overline{\Conf(D_n)}$, which is diﬀeomorphic to the associahedron(see, e.g., \cite{DS99}). The boundary $\overline{\Conf(D_n)}\backslash\Conf(D_n)$ of $\Conf(D_n)$ consists of trees of polygons.

Specifically $D_2$ is the disk with two punctures on the boundary, and we identify it bi holomorphicly with $[0,1]\times \mathbb{R}$. For each $n\ge3$ we say that $p_{n,1}$ is an outgoing point while the others are incoming points. To set up the Fukaya category we want to work with we need to identify a neighborhood of each positive puncture biholomorphicly with $Z_+=[0,1]\times[0,+\infty)$, and a neighborhood of each negative point biholomorphicly with $Z_-=[0,1]\times(-\infty,0]$. The identifications must satisfy some consistency conditions in terms of  \cite[Section 9(c)]{Sei08} and \cite[Section 9(g)]{Sei08}, so that we have a parametrization of neighborhoods of boundary strata of $\overline{\Conf(D_n)}$ in terms of lower moduli spaces and some gluing parameters.

\begin{definition}\label{complex}
	Fix a K\"{a}hler structure $(\mathfrak{j}, \eta)$ over $\Sigma$, a finite collection of points $\{z_i\}_{i=1}^m \subset\Sigma\backslash\cup_{i=1}^m\bbbeta_i$ such that each connected component of $\Sigma\backslash\cup_{i=1}^m\bbbeta_i$ contains at least one $z_i$, and an open set $V$ with \[D\cup\{z_i\}^m_{i=1}\times \Sym^{g-1}(\Sigma)\cup\Delta\subset V\subset \Sym^g (\Sigma),\]
	where $D$ is the discrete part of the fixed point of $\tau$ in $\Sym^g(\Sigma)$, and $\Delta\subset\Sym^g(\Sigma)$ is the fat diagonal.
An admissible collection of almost-complex structures up to $m$ consists of

$\bullet$ a choice of a smooth  family $\{J^2_q\}_{q\in D_2}$ of $(\mathfrak{j},\eta,V)$-nearly symmetric(in the sense of \cite[Definition 3.1]{OS04a}) almost complex structures  on $\Sym^g(\Sigma)$ parameterized by $[0,1]$(or equivalently an $\mathbb{R}$-invariant family parameterized by $D_2$).

$\bullet$ a choice of a smooth family $\{J^n_{(j,q)}\}_{(j,q)\in\Conf(D_n)\times D_n}$ of $(\mathfrak{j},\eta,V)$-nearly symmetric(in the sense of \cite[Definition 3.1]{OS04a}) almost-complex structures on $\Sym^g(\Sigma)$ parameterized by $\Conf(D_n)\times D_n$ for each $n\ge3$, satisfying the following conditions:
\begin{enumerate}
	\item[(J-1)] For $n\ge 3$, each almost-complex structure $\{J^n_{(j,q)}\}_{(j,q)\in\Conf(D_n)\times D_n}$ agrees with $J^2$ near the punctures of $D_n$, in the sense that for each $j\in\Conf(D_n)$, on a neighborhood $U$ of a positive puncture and the identification $U\isom Z_+$ specified before the definition, $J^n_{(j,(s,t))}=J^2_{s}$. The condition for the negative punctures is similar.
	\item[(J-2)]Suppose that $\{j_\alpha\}\subset\Conf(D_n)$ is a sequence converging to a point $j_\infty\in \partial\Conf(D_n)$.
	For notational simplicity, suppose $j_\infty$ lies in the codimension-one boundary, and hence corresponds to a point $(j_{\infty,1},j_{\infty,2})\in\Conf(D_{m+1})\times\Conf(D_{n-m+1})$. Then the complex structures $J_{(j_\alpha,\cdot)}$ (as families parameterized over $D_n$) are required to converge to the complex structure $J_{j_\infty,1}\sqcup J_{j_\infty,2}$(as a family parameterized by $D_{m+1}\sqcup D_{n-m+1}$).
	(Convergence of the $J_{j_\alpha}$ means the following. As $\alpha\to\infty$, certain arcs in $D_{m+1}$ collapse. Over neighborhoods of these arcs, the complex structures $J_{j_\alpha}$ should be obtained by inserting longer and longer necks(as in \cite[Section 3.4]{BEH03}). Outside these neighborhoods, we require convergence in the $C^\infty$-topology.)
	We require the analogous compatibility condition for the higher-codimension boundary of $\Conf(D_n)$, as well.
	
	We denote the space of all admissible collection of almost-complex structures up to $m$ by $\mathcal{J}_m$. If a symplectic involution $\tau$ preserving all Heegaard tori , we denote $\mathcal{J}_m^\tau\subset\mathcal{J}_m$ the $\tau$-invariant part.
\end{enumerate}
\end{definition}

\begin{remark}
	
	\begin{enumerate}
		\item[1)]We didn't elaborate some of the condition carefully here but refer the reader to \cite[Section 9]{Sei08} for more details. Condition (J-1) and (J-2) are essentially the consistency condition of perturbation data as defined in \cite[Section 9(i)]{Sei08}. 
		\item[2)]In \cite[Lemma 3.14]{OS04a} the authors required the almost complex structures to lie in a small neighborhood of $\Sym({\mathfrak{j}})$ of some specific $\mathfrak{j}$ to avoid bubbling and boundary degeneration of holomorphic curves, we drop this requirement in our context since we will always work in the hat version, where all holomorphic curves satisfy $n_z=0$. 
	\end{enumerate}
\end{remark}

\begin{definition}\label{mapcondition}
	For $1\le i_1<i_2<\hdots<i_n\le m$, $x_{l,l+1}\in T_{\bbbeta_{i_l}}\cap T_{\bbbeta_{i_{l+1}}}$ for $1\le l\le n$, for any Whitney $n$-gon $\phi\in\pi_2(x_{1,2},\hdots, x_{n-1,n}, x_{n,1})$, we denote the space $\mathcal{B}^{k,p}_\delta(\phi)$ the space of $W^{k,p}$ maps  $u$ from $D_n$ into $\Sym^g(\Sigma)$ such that
	\begin{enumerate}
		 \item[1)]$u(e_l)\subset T_{i_l}$ for each $1\le l\le n$.
		 \item[2)]$u$ extends continuously to the unit disk by mapping $p_{l,l+1}$ to $x_{l,l+1}$, and the extension is homotopic to $\phi$.
		 \item[3)]The restrictiosn of $u$ near the punctures lie in some weighted Sobolev space $W^{k,p}_\delta$(see \cite[Section 3.2]{OS04a}). 
		 \end{enumerate}

		 Fix an admissible collection of almost-complex structures $\{J^n\}$, and $\phi$ as above, define the moduli space $\mathcal{M}_{\{J^n\}}(\phi)$ to be the set of all pairs $(u,j)\in \mathcal{B}^{k,p}_\delta\times\Conf(D_n)$ such that $u$ is a $(j,J^n)$ holomorphic curve. For every such $(u,j)$, we have a linear map \[D\overline{\partial}_{u,j}:T_u\mathcal{B}^{k,p}_\delta\oplus T_j\Conf(D_n)\to W^{k,p}_\delta(u^*T\Sym^g(\Sigma)\tensor_{J^n,j}\wedge^{0,1}T^*D_n)\]($n=2$ is a special case where we omit the conformal structure factors in the source and target).
		 We say that $\mathcal{M}_{J^n}(\phi)$ is transversely cut out if for all $(u,j)\in \mathcal{M}_{\{J^n\}}(\phi)$, the map $D{\overline{\partial}}_{u,j}$ is surjective.
		 We say $\{J^n\}$ achieves transversality for the Heegaard pre-datum $(\Sigma,\{\bbbeta_i\}_{1\le i\le m},z)$ if for any choices $n\le m$, $1\le i_1<\hdots<i_n\le m$, $x_{l,l+1}\in T_{\bbbeta_{i_l}}\cap T_{\bbbeta_{i_{l+1}}}$ for $1\le l\le n$, for any Whitney $n$-gon $\phi\in\pi_2(x_{1,2}\hdots, x_{n-1,n}, x_{n,1})$ such that $\mu(\phi)\le 4-n$, $n_z(\phi)=0$, the space $\mathcal{M}_{J^n}(\phi)$ is transversely cut out.
\end{definition}

\begin{remark}
	Sometimes we omit the subscripts and superscripts in the notations $\mathcal{M}_{J^n}(\phi)$ and $\mathcal{B}^{k,p}_\delta$ when they are clear in the context.
\end{remark}

\begin{proposition}\label{transversality}
	
	Generic choices of $\{J^n\}_{1\le n\le m}\in\mathcal{J}^\tau_m$ achieve transversality for the Heegaard datum $(\Sigma,\{\bbbeta_i\}_{1\le i\le m},z)$.
\end{proposition}
	{\bf{Proof}}:
	
	\
	For a non-constant homotopy class $\phi\in\pi_2(x_{1,2},\hdots, x_{n,1})$, any holomorphic $n$-gon $u\in\mathcal{M}(\phi)$ is not contained in $\Fix(\tau)$ by Lemma \ref{discrete}, thus generic elements in $\mathcal{J}^\tau_m$ achieves transversality for the moduli space $\mathcal{M}(\phi)$ by the same argument as in \cite[Prop 5.13]{KS02}, \cite[Prop13.5]{LIP06}, \cite[Proof of Corollary 1.12(Page 1183)]{HLS16}. If $\phi$ is constant, then any holomorphic curves representing $\phi$ must also be constant by the energy estimation \cite[Equation (7)]{OS04a}. The transversality of the constant curve $u$ is guaranteed by the local arrangement Condition \ref{condition} and that the almost complex structures are fixed near $D$(guaranteed by the $(\mathfrak{j},\eta,V)$-nearly symmetric condition in \ref{complex})(Actually constant curves have Maslov index equal to $0$ by Condition \ref{condition}, thus if a constant curve with Maslov index less or equal than $4-n$ appears, then $4-n\ge 0$, which means we only need to consider constant curves when $n=2,3,4$).
	\qed

    \begin{proposition}\label{same}
    	
    	Fix a choice of $\{J^n\}_{1\le n\le m}\in \mathcal{J}^\tau_m$ satisfying the conclusion of Lemma \ref{transversality}, then for any sufficiently small perturbations $\{\bbbeta'_i\}_{1\le i\le m}$ of ${\{\bbbeta_i\}}_{1\le i\le m}$, $\{{J^n}\}_{1\le n\le m}$ achieves transversality for the Heegaard datum $(\Sigma,{\{\bbbeta'_i\}}_{1\le i\le m},z)$. Furthermore, for every $1\le i_1<\hdots<i_n\le m$, $x_{l,l+1}\in T_{\bbbeta_{i_l}}\cap T_{\bbbeta_{i_{l+1}}}$ for $1\le l\le n$, for any Whitney $n$-gon {$\phi'\in\pi_2(x_{1,2}\hdots, x_{n-1,n}, \\x_{n,1})$} with $\mu(\phi)\le3-n, n_z(\phi)=0$, if we denote $\phi'$ the corresponding homotopy class with respect to the perturbed diagram $(\Sigma,\{\bbbeta'_i\}_{1\le i\le m},z)$, we have $\mathcal{M}(\phi)\isom \mathcal{M}(\phi')$.
    \end{proposition}

    {\bf{Proof}:}
    
    We first prove the case for $n\ge3$.
    We will work in the configuration space $\mathcal{C}_n(\phi)=\mathcal{B}_{\mathcal{P}}(\phi)\times\Conf(D_n)$, where $\mathcal{P}$ is a finite dimensional manifold parameterizing sets of curves $\{\bbbeta'_i\}_{1\le i\le m}$ which are $C^\infty$-small perturbations of $\{\bbbeta_i\}_{1\le i\le m}$, and $\mathcal{B}_\mathcal{P}(\phi)\subset\mathcal{B}(\phi)\times\mathcal{P}$ is the space consists of pair $(u,\epsilon)$ such that $u$ satisfies conditions as in  Definition \ref{mapcondition}, only with the boundary condition altered by the perturbation $\epsilon$. Consider the Banach vector bundle $\mathcal{E}$ over $\mathcal{C}_m(\phi)$ whose fiber over $(u,\epsilon,j)$ is $W^{k,p}_\delta(u^*T\Sym^g(\Sigma)\tensor_{J^n,j}\wedge^{0,1}T^*D_n)$. Then we have the universal Cauchy-Riemann operator $\bar{\partial}:\mathcal{C}_m(\phi)\to \mathcal{E}$. Denote $\mathcal{M}_{\epsilon}(\phi)\subset\mathcal{B}(\phi)\times \{\epsilon\}\times\Conf(D_n)$ the zeros of $\bar{\partial}$ restricted to $\mathcal{B}(\phi)\times \{\epsilon\}\times\Conf(D_n)$. If we denote $*\in \mathcal{P}$ the zero perturbation, then the moduli space $\mathcal{M}_{*}(\phi)\subset\mathcal{B}(\phi)\times \{*\}\times\Conf(D_n)$ is transversely cut out as required by the condition.
    
    {\bf{Claim:}}
    Consider a sequence of perturbations ${\{\epsilon_\alpha\}}_{\alpha=1}^\infty$ that tends to constant, as well as a sequence of $J_n$-holomorphic $n$-gons $\psi_n\in{\mathcal{M}}_{\epsilon}$, we claim that $\psi_n$ converge to some element in $\mathcal{M}_{*}(\phi)\subset \mathcal{B}_{\mathcal{P}}(\phi)$. 
    
    By Gromov's compactness theorem, the sequence of holomorphic $n$-gons converges to some stable holomorphic polygon, and some components may be constant maps. Denote $\varphi_1,\hdots, \varphi_s$ these components, and $n_1\hdots n_s$ the number of their edges, then by the transversality guaranteed in Lemma \ref{transversality}, $\mu(\varphi_i)\ge3-n_i$. Given condition \ref{condition}, every constant polygon has Maslov index $0$, thus $\mu(\varphi_1)+\hdots+\mu(\varphi_s)=\mu(\phi)$. We also have $n_1+\hdots +n_s\le n+(s-1)$, since degeneration along an arc produces at most two more edges. Then above arguments give us inequalities:
    \begin{align*}
    	&3-n_1+\hdots+3-n_s\le 3-n\\
    	&n_1+\hdots +n_s\le n+2(s-1)
    \end{align*} 
    But these imply $s\le1$, so $s=1$, that is the sequence of holomorphic $k$-gons converges to a genuine holomorphic $k$-gon in the unperturbed setting, and this proves the claim.
    
    Specifically, the claim implies that the moduli space $\mathcal{M}_*(\phi)$ is compact. Now by the claim, for every neighborhood $\mathcal{U}\supset\mathcal{M}_*(\phi)$ in $\mathcal{C}_n(\phi)$ there exist a neighborhood $\mathcal{W}_\phi$ of $*$ in $\mathcal{P}$ such that for every $\epsilon\in \mathcal{W}$, $\mathcal{M}_{\epsilon}(\phi)\subset \mathcal{U}$. For every $\phi$ satisfying the assumption of the Proposition,  we take $\mathcal{U_\phi}$ small enough so that $\mathcal{M}(\phi)\cap\mathcal{U}$ is everywhere transversely cut out and take $\mathcal{W_\phi}$ accordingly. Notice that by the weak-admissibility guaranteed by Lemma \ref{admissible}, there are only finitely many such $\phi$ which can support holomorphic polygon, thus we can take $\mathcal{W}$ to be a finite intersection of $\mathcal{W}_\phi$, so that for every $\epsilon\in\mathcal{W}$, $\{J_n\}$ achieves transversality for the Heegaard datum perturbed by $\epsilon$.
    
    For every $\epsilon\in \mathcal{W}$, pick a path $\gamma$ connecting $*$ and $\epsilon$. By the same argument as in the claim we can prove the associated moduli space $\mathcal{M}_\gamma(\phi)$ is compact(in the claim we only used $\mathcal{M}_*(\phi)$ is transversely cut out for all $\phi$ with $n_z=0$ and $\mu(\phi)\le 4-n$).  Moreover $\mathcal{M}_{\gamma}(\phi)$ is smooth and the projection $\mathcal{M}_\gamma(\phi)\to[0,1]$ is a submersion (guaranteed by the everywhere transversality over $\mathcal{W}$), thus it is a trivial cobordism, and this proves the last assertion.
    
    The $n=2$ case is similar. By the same argument as before we know $\mathcal{M}_\epsilon(\phi)$ is transversely cut out for $\epsilon$ small enough(notice that the transversality property is invariant under $\mathbb{R}$-action), so we only need to take care of $\mu(\phi)=1$ case. The only issue is that the cobordism $\mathcal{M}_{\gamma}(\phi)$ is non-compact but compact after modulo the $\mathbb{R}$ action. We prove the $\mathbb{R}$-action on $\mathcal{M}_\gamma(\phi)$ is free and proper, so the quotient $\mathcal{M}_{\gamma}(\phi)/\mathbb{R}$ is a manifold. The free part is obvious. Pick an arbitrary metric $d$ on $\Sym^g(\Sigma)$. If there is a sequence $(a_i,u_i)\in \mathbb{R}\times\mathcal{M}_\gamma(\phi)$ and such that $(a_i\cdot u_i,u_i)$ lies in some compact subset of $\mathcal{M}_\gamma(\phi)\times\mathcal{M}_\gamma(\phi)$. By compactness there are $v,w\in\mathcal{M}_\gamma(\phi)$ such that $d(v,u_i)\rightarrow0$, $d(w,a_i\cdot u_i)\rightarrow0$, where here $d(u',v')=\sup_{p\in[0,1]\times \mathbb{R}}d(u'(p),v'(p))$. Thus, $d(w,a_i\cdot v)\rightarrow0$. If $\{a_i\}$ is unbounded, then $w$ is constant, a contradiction! This proves the action of $\mathbb{R}$ on $\mathcal{M}_\gamma(\phi)$ is proper, so the quotient space is a manifold. As in the previous paragraph the projection $\mathcal{M}_\gamma(\phi)\to [0,1]$ is a submersion. The $n=2$ case then follows as above.
\qed

		\section[exacttriangle]{Proof of the Exact Triangle}\label{proof}

We denote the Heegaard pre-datum constructed in Section \ref{1.2} by $(\Sigma,\pralpha, \prbeta,\prbeta',\\ \prbeta'',z)$, and its small perturbation supported near $\bold{p}$ by $(\Sigma, \beta,\beta',\beta'',z)$(on $\Sigma_{g-1}$ component it looks like the left of Figure (\ref{handleslide}) omitting the blue dashed line). For convenience we shift the grading of $\HFh(\#^l S^1\times S^2)$ so its top grading is $0$.

\subsection{Twisted complexes}

\

To simplify the necessary algebra we introduce the language of twisted complexes. The content of this subsection is from \cite[Section 2.4]{Nah25a}.

We work with non-unital $A_\infty$-categories over $\Field=\Z$. Our conventions are slightly different from \cite{Sei08}: we use the homological convention instead of the cohomological convention, and $A_\infty$-categories are not necessarily (homologically) graded. However, in some special cases the $A_\infty$ category in this article will be graded: we say that an $A_\infty$ category is graded if $\Hom(\alpha,\beta)$ are absolutely $\Z$-graded for every objects $\alpha$ and $\beta$, and the composition maps $\mu_d$ has degree $d-2$.

\begin{definition}
	Let $\mA$ be an $A_\infty$ category. The \emph{additive enlargement} $\Sigma\mA$ is defined as follows. Objects are formal direct sums
	\[\bigoplus_{i\in I}V_i\tensor\alpha_i,\]
	where $I$ is a finite set, $\{\alpha_i\}$ is a family of objects in $\mA$, and $\{V_i\}$ is a family of finite-dimensional vector spaces. Morphisms are defined a combination of morphisms between the vector spaces and morphisms in $\mA$:
	\[\Hom(\bigoplus_{i\in I}V_i\tensor\alpha_i,\bigoplus_{j\in J}W_j\tensor\beta_j)=\bigoplus\Hom_\Field(V_i,W_j)\tensor\Hom_\mA(\alpha_i,\beta_j)\].
	
	Similarly, compositions $\mu_d$ are defined by combining the ordinary composition of maps between
	vector spaces and morphisms in $\mA$. The additive enlargement $\Sigma\mA$ forms an $A_\infty$-category.
	
	\
	
	A \emph{twisted complex} in $\mA$ consists of an object $\underline{\alpha}\in \operatorname{Ob}(\Sigma\mA)$ together with a differential ${\delta_{\underline{\alpha}}}\in
	{\Hom_{\Sigma\mA}}{(\underline{\alpha},\underline{\alpha})}$, such that
	\[\sum_{n\ge1}{\mu_n}^{\Sigma\mA}{({\delta_{\underline{\alpha}}},\hdots,{\delta_{\underline{\alpha}}})}= 0.\]
	Seidel assumes that the differential $\delta_{\underline{\alpha}}$ is strictly lower triangular with respect to a filtration on
	$\underline{\alpha}$ to ensure that the above sum is finite.
	
	\
	
	Twisted complexes in $\mA$ also form an $\mA_\infty$-category $\operatorname{Tw}\mA$. Morphisms between twisted complexes are the same as before,
	\[{\Hom_{\operatorname{Tw}}}(\alpha,\delta_{\underline{\alpha}}),(\beta,\delta_{\underline{\beta}}) = \Hom_{\Sigma\mA}(\underline{\alpha},\underline{\beta}),\]
	but the compositions are different:
\[
\mu_d^{\operatorname{Tw}\mA}(\varphi_1,\hdots,\varphi_d)
=
\sum_{i_0,\hdots,i_d\ge 0}
\mu^{\Sigma\mA}_{d+i_0+\hdots+i_d}
\Bigl(
\overbrace{
	\delta_{\underline{\alpha_0}},\hdots,
	\delta_{\underline{\alpha_0}}
}^{i_0},
\underline{\varphi_1},
\overbrace{
	\delta_{\underline{\alpha_1}},\hdots,
	\delta_{\underline{\alpha_1}}
}^{i_1},
\underline{\varphi_2},
\hdots,
\underline{\varphi_d},
\overbrace{
	\delta_{\underline{\alpha_d}},\hdots,
	\delta_{\underline{\alpha_d}}
}^{i_d}
\Bigr).
\]
\end{definition}

\begin{remark}
	\begin{enumerate}
		\item[1)]
		In this article the $A_\infty$ category $\mA$ will always be given by some weakly admissible Heegaard datum and some choices of almost complex structures. All the vector space $V_i$ will be $\Field$, and in the category $\operatorname{Tw}\mA$ we will only care about $\mu^{\operatorname{Tw}\mA}_1$ 
	\end{enumerate}
\end{remark}

\begin{example}
	The following diagram forms a twisted complex if and only if $\mu_1(f)=0$, $\mu_1(g)=0$ and $\mu_2(f,g)=\mu_1(h)$.

	\[\begin{tikzcd}
		\underline{\beta}=\beta_1\arrow[r,"f"]\arrow[rr,bend right,"h"]&\beta_2\arrow[r,"g"]&\beta_3
	\end{tikzcd}\]
\end{example}

\subsection{Some strong equivalences of the Heegaard datum $(\Sigma,\prbeta,\prbeta',\prbeta'',z)$}

\

The main goal of this section is to prove Theorem \ref{main}. For this purpose we compare our cobordism maps with the traditional setting of exact triangles in \cite{OS04b} and \cite{OS05}(see diagrams (\ref{compare1}) and (\ref{compare2})). First we do handleslides and isotopy on the genus $2$ component $\Sigma_{g-1}$ as in Figure (\ref{handleslide}) and the genus $1$ components $\Sigma_i$ for $1\le i \le g-2$ as follows. 

\begin{figure}
	\begin{subfigure}{0.4\textwidth}
		\includegraphics[width=\textwidth]{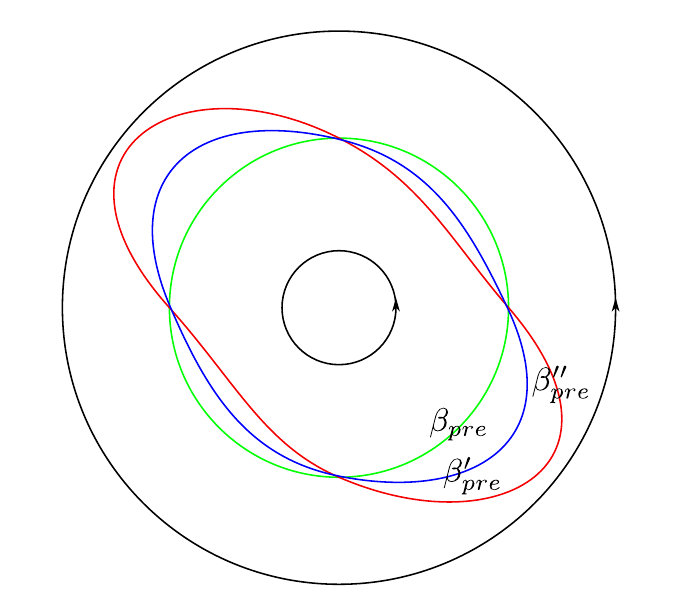}
	\end{subfigure}
	\begin{subfigure}{0.4\textwidth}
	  	\centering
	  \includegraphics[width=\textwidth]{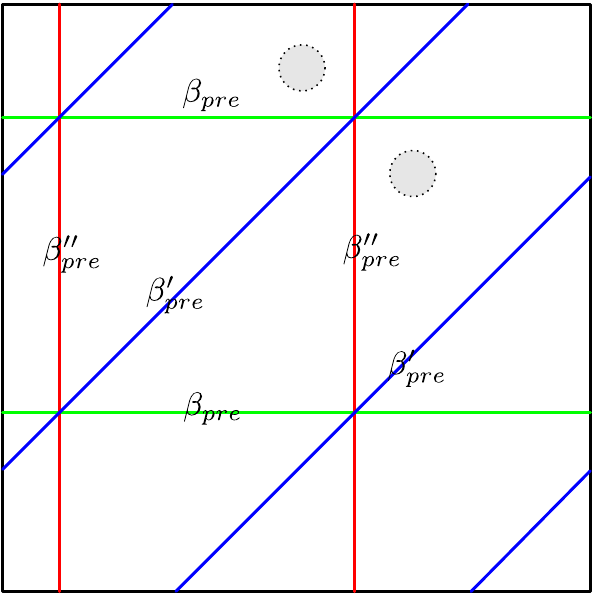}
\end{subfigure}
\caption{Left: The $\Sigma_i$ component($1\le i\le g-2$) of the Heegaard diagram $(\Sigma,\prbeta,\prbeta',\prbeta'',z)$, where we identify the two boundaries as indicated. Right: The $\Sigma_{g-1}$ component($1\le i\le g-2$) of the Heegaard diagram $(\Sigma,\prbeta,\prbeta',\prbeta'',z)$, where we do $0$-surgery at the two grey disks and identify the parallel edges of the square. It can also be seen as attaching a cylinder to the bottom and top of the right of Figure (\ref{double}).}\label{square}

\end{figure}

\begin{figure}
	\centering
\begin{subfigure}{0.3\textwidth}
		\centering
	\includegraphics[width=\textwidth]{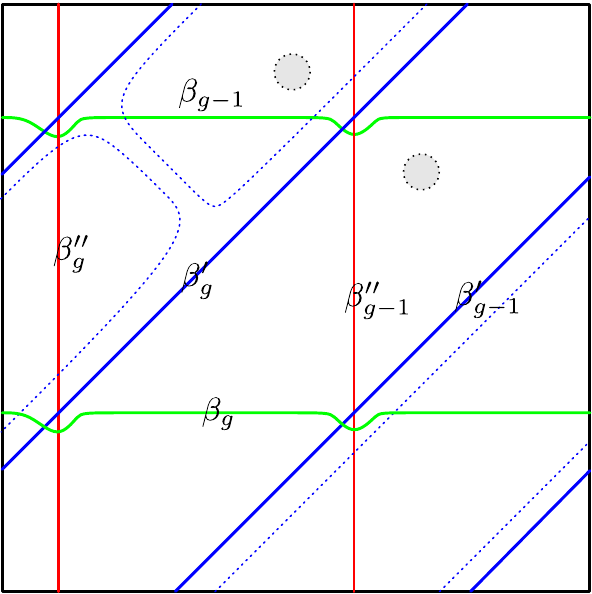}
	\end{subfigure}
	\hfill
	\begin{subfigure}{0.3\textwidth}
		\centering
	\includegraphics[width=\textwidth]{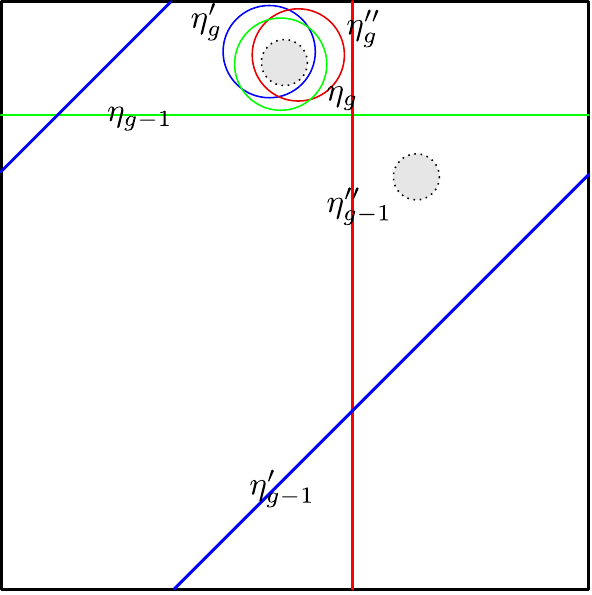}
\end{subfigure}
\hfill
\begin{subfigure}{0.3\textwidth}
	\centering
	\includegraphics[width=\textwidth]{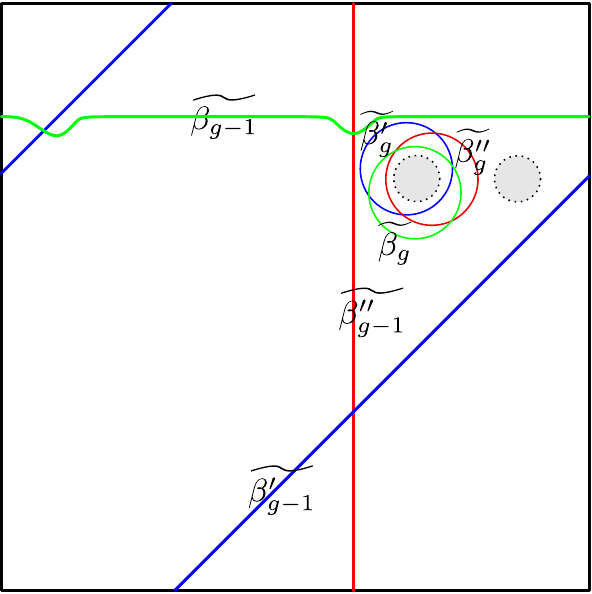}
\end{subfigure}
\caption{The strong equivalences on the $\Sigma_{g-1}$ component. Left: The solid curves represent a small perturbation of the pre-involutive diagram $(\Sigma,\prbeta,\prbeta',\prbeta'',z)$, while the blue dashed line represents the result of handlesliding of $\beta_{g}'$ over $\beta_{g-1}'$. Middle: The result of handlesliding $\beta_{g}$ over $\beta_{g-1}$, $\beta_{g}'$ over $\beta_{g-1}'$ and $\beta_{g}''$ over $\beta_{g-1}''$ followed by some further isotopies. Right: The result of handlesliding $\eta_{g}$ over $\eta_{g-1}$, $\eta_{g}''$ over $\eta_{g-1}''$.} \label{handleslide}
\end{figure}

First we work over the genus $2$ component. To simplify notations, in the this paragraph the range of index $i$ will be $\{g-1,g\}$. We draw the right of Figure (\ref{double}) on a square as in the right of Figure (\ref{square}), where we identify the boundaries of the grey disks to get a genus $2$ orientable surface, and the basepoint $z$ lies inside the cylinder connecting the boundaries of the grey disks. Then we perform small isotopy of the pre-involutive Heegaard datum (right of Figure (\ref{square})) to get the involutive Heegaard datum $(\Sigma,\beta,\beta',\beta'',z)$(which are indicated in the left of Figure (\ref{handleslide}) omitting the dashed blue curve). Then we handleslide $\beta'_g$ curves as indicated in the left of Figure (\ref{handleslide}) over $\beta'_{g-1}$ curve, and similar handlesildes of $\beta_g$ and $\beta_g''$ curves over $\beta_{g-1}$ and $\beta_{g-1}''$ curves respectively. After some further isotopies we get a diagram as in the middle of Figure (\ref{handleslide}), and we denote the new curves $\eta_i$, $\eta_i'$, $\eta_i''$ as shown in the middle of Figure (\ref{handleslide}). Unfortunately the middle of Figure (\ref{handleslide}) is not weakly admissible, to mend this we handleslide $\eta_{g}$ over $\eta_{g-1}$, $\eta_{g}''$ over $\eta_{g-1}''$ respectively to get the curves  $\tbeta_i$, $\tbeta_{i}'$, $\tbeta_{i}''$  as in the right of Figure (\ref{handleslide}). 

For the genus $1$ parts, $1\le i\le g-1$, we simply isotope $\beta_i$, $\beta'_i$ and $\beta''_i$ to a position such that they are small Hamiltonian isotopic of each other and every two of them intersect transversely in two points, and set the new curves to be $\tbeta_i$, $\tbeta'_i$ and $\tbeta''_i$ in the corresponding components. 

We assume the curves in the Heegaard pre-datum $(\Sigma,\prbeta,\prbeta',\prbeta'',\beta,\beta',\beta'',\tbeta, \tbeta',\tbeta'',z)$ are in general positions.
 
Now the Heegaard diagram $(\Sigma_{g-1}, \{\tbeta_{g-1}, \tbeta_{g}\}, \{\tbeta'_{g-1}, \tbeta'_{g}\},\{\tbeta''_{g-1}, \tbeta''_{g}\},z)$ can be viewed as a connected sum $(\Sigma^{(1)}_{g-1}, \tbeta_{g-1}, \tbeta'_{g-1}, \tbeta''_{g-1}, z^{(1)})$ and $(\Sigma^{(2)}_{g-1}, \tbeta_{g}, \tbeta'_{g}, \tbeta''_{g}, z^{(2)})$ at the basepoints, where $\tbeta_{g}, \tbeta'_{g}, \tbeta''_{g}$ are nontrivial curves on a  torus that are small Hamiltonian isotopy of each other and every two of them intersect transversely in two points, while $\tbeta_{g-1}, \tbeta'_{g-1}, \tbeta''_{g-1}$ are curves on the torus that has $0$, $1$, $\infty$ slope respectively. Now the Heegaard triple $(\Sigma,\tbeta,\tbeta',\tbeta'',z)$ is in the setting of exact triangle we are familiar with(see \cite[Theorem 4.7, Figure 8, Figure 9]{OS05}).

\subsection{Choices of almost complex structures and the elements $\Theta$}\label{choices}
The first preparation for the final argument we will do is to choose an almost complex structure we want to work with. Now we have a Heegaard pre-datum $(\Sigma,\prbeta,\prbeta',\prbeta'',\beta,\beta',\beta'',\\ \tbeta, \tbeta',\tbeta'',z)$, to use $A_\infty$ relations we need to choose an admissible almost complex structure achieving transversality. We do this as follows. We first choose an equivariant admissible almost complex structure $\{J_n\}$ achieving transversality for the Heegaard pre-datum $(\Sigma,\prbeta,\prbeta',\prbeta'',\tbeta, \tbeta',\tbeta'',z)$ as in Lemma \ref{transversality}(notice that by the general position any holomorphic polygon with boundary on $T_{\tbeta}, T_{\tbeta'}$ or $T_{\tbeta''}$ must be non-$\tau$-invariant, so the argument in Lemma \ref{transversality} works), and then we choose the perturbations $\tbeta, \tbeta',\tbeta''$ small enough so that the result of Proposition \ref{same} holds for the admissible almost complex structure $\{J^n\}$ and the perturbations $\tbeta, \tbeta',\tbeta''$.
Notice that we are considering the hat version of Heegaard Floer homology, so we do not need to worry about the case of boundary degenerations and sphere bubbling, that is we can freely use the $A_\infty$ relations in the following arguments. 

The second preparation is to choose cycles in $\CFh(\beta,\beta')$, $\CFh(\beta',\beta'')$, $\CFh(\tbeta,\tbeta')$ and $\CFh(\tbeta',\tbeta'')$ representing the generators in the corresponding homology in the $0$ grading, which we will denote by $\Theta_{\beta\beta'}$, $\Theta_{\beta'\beta''}$, $\Theta_{\tbeta\tbeta'}$ and $\Theta_{\tbeta'\tbeta''}$ respectively. The chains $\Theta_{\tbeta\tbeta'}$ and $\Theta_{\tbeta'\tbeta''}$ are easy to choose since the Heegaard diagrams $(\Sigma,\tbeta,\tbeta')$ and $(\Sigma,\tbeta',\tbeta'')$ are so simple that there is only one choice for the generator in the $0$ grading. The Heegaard diagrams $(\Sigma,\beta,\beta')$ and $(\Sigma,\beta',\beta'')$ have more generators in the $0$ grading, for each of them we choose $\Theta$ to be a cycle representing the generators for the corresponding homology in the $0$ grading (the choice may vary with the almost complex structure we chose). Notice that the generator $\Theta_{\beta\beta'}$, $\Theta_{\beta'\beta''}$ we chose "nearly" $\tau$-invariant, that is they are close to invariant cycles in the corresponding pre-involutive diagrams before the perturbations: For the $\Sigma_{g-1}$ component, every intersection points on $\Sigma_{g-1}$ is nearly $\tau$-invariant(see the left of Figure (\ref{handleslide})(omitting the dashed lines)); For the $\Sigma_i(1\le i \le g-2)$ component, the $\beta_i, \beta_i'$ and $\beta_i''$ curves are small perturbations of those in the left of Figure (\ref{square}), and the generators in the $0$ grading should be chosen to contain the points close to the multiple intersections thanks to our local arrangement at the fixed points. This will be important in the proof of the main theorem \ref{proofmain} since we want the maps we constructed between $\CFh(\alpha,\beta), \CFh(\alpha,\beta)$ and $\CFh(\alpha,\beta'')$ to be equivariant, see the top line of diagram (\ref{compare1}).

\subsection{Proof of the main theorem}
\begin{lemma}\label{lem}
There exists an element $\theta\in \CFh(\beta,\beta'')$ such that $\mu_1(\theta)=\mu_2(\Theta_{\beta\beta'},\Theta_{\beta'\beta''})$, i.e. the following is a twisted complex:
	\[
\begin{tikzcd}[column sep=large, row sep=large]
	\beta \arrow[r,"\Theta_{\beta\beta'}"]\arrow[rr,bend left,"\theta"]&\beta' \arrow[r,"\Theta_{\beta'\beta''}"]&\beta''\\
\end{tikzcd}
\]
\end{lemma}
{\bf{Proof:}}

\

The trick is to compare this with the $\tbeta, \tbeta', \tbeta''$ as in the proof of Proposition \ref{iso}, and use the handleslide invariance of triangle maps. To be more specific, since $\tbeta, \tbeta', \tbeta''$ is obtained from $\beta, \beta', \beta''$ by handleslides, by of \cite[Proposition 4.6]{OS06}, the diagram
\begin{equation}\label{lalala}
\begin{tikzcd}[column sep=huge]
	\CFh(\beta,\beta')\arrow[r,"{\mu_2(\cdot,\Theta_{\beta'\beta''})}"]\arrow[d,"{\Psi_{\beta\to\tbeta,\beta'\to\tbeta'}}"]&\CFh(\beta,\beta'')\arrow[d,"{\Psi_{\beta\to\tbeta,\beta''\to\tbeta''}}"]\\
	\CFh(\tbeta,\tbeta')\arrow[r,"{\mu_2(\cdot,\Theta_{\tbeta'\tbeta''})}"]&\CFh(\tbeta,\tbeta'')
	\end{tikzcd}
\end{equation}
commutes up to chain homotopy, where the two vertical maps are the chain homotopy equivalences induced by strong equivalences of Heegaard diagrams, and $\Theta_{\beta'\beta''}$ and $\Theta_{\tbeta'\tbeta''}$ are representatives of the top generators of $\HFh(\beta',\beta'')$ and $\HFh(\tbeta',\tbeta'')$ chosen in Subsection \ref{choices}. As shown in Figure (\ref{handleslide}), the $\tbeta, \tbeta', \tbeta''$ curves are in the same setting as \cite[Proposition 9.5]{OS04b}, so we have that $\mu_2(\Theta_{\tbeta\tbeta'},\Theta_{\tbeta'\tbeta''})$ is null-homologous, combining this with diagram (\ref{lalala}), $\mu_2(\Theta_{\beta\beta'},\Theta_{\beta'\beta''})$ is also null-homologous, and the lemma follows.
\qed

\

The technical part of this section is the following one:
\begin{proposition}\label{iso}
	Given a Heegaard data $(S, \gamma,\gamma',\gamma'',\tgamma,\tgamma',\tgamma'',z)$ which is weakly admissible, such that:
	\begin{enumerate}
	  \item[1)]Each of $(S,\gamma,\gamma',z)$, $(S,\gamma',\gamma'',z)$,$(S,\gamma,\gamma'',z)$, $(S,\tgamma,\tgamma',z)$, $(S,\tgamma',\tgamma'',z)$, $(S,\tgamma,\tgamma'',z)$ represents the three manifold $\#^{l}S^1\times S^2$ for some $l$.
	  \item[2)] There exists an chain $\theta\in \CFh(\gamma,\gamma')$(and $\ttheta\in \CFh(\tgamma,\tgamma')$) such that $\mu_1(\theta)=\mu_2(\Theta_{\gamma\gamma'},\Theta_{\gamma'\gamma''})$\\(and $\mu_1(\ttheta)=\mu_2(\Theta_{\tgamma\tgamma'},\Theta_{\tgamma'\tgamma''})$). In other word, the following is a twisted complex:
	  \begin{equation}\label{c}
	  	\begin{tikzcd}[column sep=large, row sep=large]
	  		\gamma \arrow[r,"\Theta_{\gamma\gamma'}"]\arrow[rr,bend left,"\theta"]&\gamma' \arrow[r,"\Theta_{\gamma'\gamma''}"]&\gamma'',\\
	  	\end{tikzcd}
	  		  	\begin{tikzcd}[column sep=large, row sep=large]
	  		\tgamma \arrow[r,"\Theta_{\tgamma\tgamma'}"]\arrow[rr,bend left,"\ttheta"]&\tgamma' \arrow[r,"\Theta_{\tgamma'\tgamma''}"]&\tgamma'',\\
	  	\end{tikzcd}
	  \end{equation}
	where $\Theta_{\gamma\gamma'}$, $\Theta_{\gamma'\gamma''}$, $\Theta_{\tgamma\tgamma'}$ and $\Theta_{\tgamma\tgamma''}$ are any chains representing the generators of $\HFh(\gamma,\gamma')$, $\HFh(\gamma',\gamma'')$, $\HFh(\tgamma,\tgamma')$ and $\HFh(\tgamma',\tgamma'')$ in the $0$ grading, respectively. 
	\end{enumerate}
	
	\
	Then we can complete $\Theta_{\gamma\tgamma}$, $\Theta_{\gamma'\tgamma'}$ and $\Theta_{\gamma''\tgamma''}$ (which, again are some representatives for the generators of the corresponding Heegaard Floer homology in the $0$ grading) into a cycle $\Theta$ between the twisted complexes $\underline{\gamma\gamma'\gamma''}$ and $\underline{\tgamma\tgamma'\tgamma''}$ by choosing the diagonal chains appropriately:
	\begin{equation}\label{cycle}
	\begin{tikzcd}[column sep=large, row sep=large]
	\gamma \arrow[r,"{\Theta_{\gamma\gamma'}}"]\arrow[d,"{\Theta_{\gamma\tgamma}}"]\arrow[rd]\arrow[rrd]\arrow[rr,bend left,"\theta"]&\gamma' \arrow[r,"{\Theta_{\gamma'\gamma''}}"]\arrow[rd]\arrow[d,"{\Theta_{\gamma'\tgamma'}}",crossing over]&\gamma''\arrow[d,"{\Theta_{\gamma''\tgamma''}}"]\\
		\tgamma\arrow[r]\arrow[rr, bend right,"\ttheta"]&\tgamma'\arrow[r]&\tgamma''
	\end{tikzcd}
	\end{equation}

\end{proposition}
	{\bf{Proof:}}

\

Given an element $\Theta\in \CFh(\twigamma, \twitgamma)$, we denote $\Theta_{(i,j)}$ the component of $\Theta$ with source on the $i$ column in the first row, and target the $j$ column in the second row($0\le i,j\le 2$). We also denote $\Theta_{\le k}$ the sum of the $\Theta_{(i,j)}$ components with $0\le j-i\le k$. 
We first construct the $(0,1)$ and the $(1,2)$ components of the cycle.

The existence of the $(0,1)$-component is equivalent to that $\mu_2(\Theta_{\gamma\gamma'},\Theta_{\gamma'\tgamma'})$ and $\mu_2(\Theta_{\gamma\tgamma},\Theta_{\tgamma\tgamma'})$ are homologous in $\CFh(\gamma,\tgamma')$. This is because that the maps $\mu_2(\cdot,\Theta_{\gamma'\tgamma'})$ and $\mu_2(\Theta_{\gamma\tgamma},\cdot)$ are both  chain homotopy equivalences between $\CFh(\gamma,\gamma')$ and $\CFh(\gamma,\tgamma')$, $\CFh(\tgamma,\tgamma')$ and $\CFh(\gamma,\tgamma')$ respectively(\cite[Theorem 2.1]{OS06}), so $\mu_2(\Theta_{\gamma\gamma'},\Theta_{\gamma'\tgamma'})$ and $\mu_2(\Theta_{\gamma\tgamma},\Theta_{\tgamma\tgamma'})$ are both homologous to the generator of $\CFh(\gamma,\tgamma')$. The assertion for $(1,2)$ component follows similarly.

Now we are left to construct the $(0,2)$ component, we do this in the following steps. We denote $\Theta'$ a chain in $\CFh({\underline{\gamma\gamma'\gamma''}}, {\underline{\tgamma\tgamma'\tgamma''}})$ with the $(0,0), (1,1), (2,2)$ components as shown in Figure (\ref{cycle}), the $(0,1)$ and $(1,2)$-components constructed as in the last paragraph, and the $(0,2)$ component set to be $0$. Now the $A_\infty$ relations reads $\mu^{Tw}_1(\mu^{Tw}_1(\Theta'))=0$ in $\CFh(\twigamma,\twitgamma)$.

Now since the $\mu^{Tw}_1(\Theta')_{\le1}$ component is $0$ by the construction of $\Theta'_{(0,1)}$ and $\Theta'_{(1,2)}$, we know that $\mu^{Tw}_1(\mu^{Tw}_1(\Theta'))_{(0,2)}=\mu_1(\mu_1^{Tw}(\Theta'))=0$, which means that $\mu^{Tw}_1(\Theta')=\mu^{Tw}_1(\Theta')_{(0,2)}$ is a cycle of grading $1$ in $\CFh(\gamma,\tgamma'')$. Since the grading $1$ summand of $\HFh(\gamma,\tgamma'')$ vanishes, it is also a boundary, so we can construct an element $\xi$ in the grading $2$ summand of $\CFh(\gamma,\tgamma'')$ such that $\mu_1(\xi)=\mu^{Tw}_1(\Theta')$, that is if we change the $(0,2)$ component of $\Theta'$ from $0$ to $\xi$ and denote the new chain $\Theta$, then $\mu^{Tw}_1(\Theta)=0$ and $\Theta$ is the cycle we want to construct in diagram (\ref{cycle}).
\qed

\begin{lemma}\label{intermediate}
	Given a Heegaard datum $(\Sigma,\delta_1,\delta_2,\delta_3,\eta_1,\eta_2,\eta_3,z)$ such that both the Heegaard data $(\Sigma,\delta_1,\delta_2,\delta_3,z)$ and $(\Sigma,\eta_1,\eta_2,\eta_3,z)$ are weakly admissible, then there exist sets of curves $\xi_1,\xi_2,\xi_3$ on $\Sigma$ such that $\xi_i$ is obtained by isotoping $\delta_i$, and both $(\Sigma,\delta_1,\delta_2,\delta_3,\xi_1,\xi_2,\xi_3,z)$ and $(\Sigma,\xi_1,\xi_2,\xi_3,\eta_1,\eta_2,\eta_3,z)$ are Heegaard data, and both of them are weakly admissible.  
\end{lemma}

{\bf{Proof:}}

 First we introduce some notations. Given a periodic domain $D$ we write $||D||_{+}$ to be its minimal non-negative multiplicities, $||D||_-$ to be the absolute value of its maximal non-positive multiplicities, $||D||_\infty=\operatorname{max}(||D||_-,||D||_+)$, and $||D||_{\pm}=\operatorname{min}(||D||_-,||D||_+)$. For $\xi_i$ curves we write $\xi_i=\{\xi_{i1},\xi_{i2},\hdots, \xi_{ig}\}$, and similar notations for other sets of curves.
 
First we set $\xi_i$ to be small Hamiltonian isotopy of $\delta_i$, then it is easy to see that the Heegaard datum $(\Sigma,\delta_1,\delta_2,\delta_3,\xi_1,\xi_2,\xi_3,z)$ is weakly admissible. 

 Next we achieve the weak admissibility for $(\Sigma,\xi_1,\xi_2,\xi_3,\eta_1,\eta_2,\eta_3,z)$. For this sake we work with domain with rational coefficients.  Choose a $\QQ$-basis $D_1\hdots,D_{n_0}$ for the periodic domains in $(\Sigma,\eta_1,\eta_2,\eta_3,z)$ and complete it into a basis of the periodic domains of $(\Sigma,\xi_1,\xi_2,\xi_3,\eta_1,\eta_2,\eta_3,z)$ by adding $D^1_{i},\hdots, D^{m_i}_i$ for $i=1,2,3$, such that (rearrange $\xi_{i1},\xi_{i2},\hdots, \xi_{ig}$ if necessary) for $1\le j, k\le m_i$, the multiplicity of $\partial D^j_i$ at $\xi_{ik}$ is $1$ if $j=k$, and $0$ else. Before winding we denote $M_i=\max_{1\le j\le m_i}(||D^j_{i}||_\infty)$, $M'=\max_{1\le i\le n_0}(||D_i||_\infty)$ and $M=\operatorname{max}(M_1,M_2,M_3)$.
 For each $i=1,2,3$ we choose set of simple closed curve $\rho_i=\{\rho_{i1},\hdots,\rho_{ig}\}$ which is a geometric dual of $\xi_i$, that is $\rho_{ik}$ intersects $\xi_{ik}$ transversely at one point, and $\rho_{ik}$ and $\rho_{il}$ are disjoint for $k\neq l$.

Now we wind the $\xi_{ij}$ curves along $\rho_{ij}$ for $1\le i\le 3, 1\le j\le m_i$ in a small neighborhood of $\rho_{ij}$ in both directions, and in the wound diagram, we still denote the domains corresponding to $D_k$, $D_i^l$ by $D_k$, $D_i^l$. For any sufficiently large $N$, after winding sufficiently many times, we can achieve that for every index $1\le i\le 3, 1\le j\le m_i$, both $N$ and $-N$ appear as multiplicities in the winding region of $D_{i}^j$. 

Now given linear combinations $E_0=\sum_{j=1}^{n_0}{q_j}{D_j}$, and $E_i=\sum_{j=1}^{m_i}{p^j_i}{D^j_i}$ for $i=1,2,3$, denote $E=E_0+E_1+E_2+E_3$, $q=\operatorname{max}_{1\le j\le n_0}(|q_j|)$, $p_i=\operatorname{max}_{1\le j\le m_i}(|p^j_i|)$, $p=\operatorname{max}(p_1,p_2,p_3)$. 
Following \cite[Lemma 4.13]{OS04a}, since $(\Sigma,\eta_1,\eta_2,\eta_3,z)$ is weakly admissible, there exists a constant $C$ such that $||E_0||_{\pm}\ge C||E_0||_\infty$. Since $||\cdot||_\infty$ and $\operatorname{max}_{1\le j\le m_i}(|p^j_i|)$ define equivalent norms, we have $||E_0||_{\pm}\ge C'p$ for some constant $C'$ not depending on $N$.
We have the following estimations:
\begin{enumerate}
	\item[1)]Assume $D=D_1^1$ achieves the maximal value in $p=\operatorname{max}(p_1,p_2,p_3)$. Without loss of generality assume $p^1_1$ is non-negative. Consider a region supported in the winding region of $D$ of multiplicity $N$, for a point $a$ in $D$ not in other winding regions, the multiplicity for $E-pD$ at $a$ is at least $-n_0qM'-(m_1+m_2+m_3)pM$. Thus, if 
	\begin{equation}\label{es1}
	Np>n_0qM'+(m_1+m_2+m_3)pM
	\end{equation}
	then $E$ has at least one positive multiplicity. By the same analysis if equation (\ref{es1}) holds then $E$ also has at least one negative multiplicity.
	\item[2)]Assume $D=D_1$ achieves the maximal value in $q=\operatorname{max}_{1\le j\le n_0}(|q_j|)$. Without loss of generality assume $p_1$ is non-negative. By the same argument as above if $||E||_{\pm}>(m_1+m_2+m_3)pM$, then $E$ has both positive and negative multiplicities. Since $||E_0||_{\pm}\ge C'p$,
	\begin{equation}\label{es2}
		C'q>(m_1+m_2+m_3)pM
		\end{equation} suffices for the existence of both positive and negative multiplicities,
\end{enumerate}

However, one of equations (\ref{es1}) and (\ref{es2}) must hold if $N$ is large enough. This proves that $(\Sigma,\xi_1,\xi_2,\xi_3,\eta_1,\eta_2,\eta_3,z)$ is weakly admissible. Since we only wind $\xi_i$ in small neighborhoods of some curves, $(\Sigma,\delta_1,\delta_2,\delta_3,\xi_1,\xi_2,\xi_3,z)$ is still weakly admissible.
\qed

\

{\bf{Proof of Theorem \ref{main}}:}\label{proofmain}
At first we do small perturbations to the Heegaard datum $(\Sigma,\beta,\beta',\beta'',z)$, by Proposition \ref{admissible} $(\Sigma,\beta,\beta',\beta'',z)$ is automatically weakly admissible. Since the Heegaard datum $(\Sigma,\beta,\beta',\beta'',\tbeta,\tbeta'\,\tbeta'', z)$ may not be weakly admissible, by lemma \ref{intermediate} we choose an intermediate Heegaard datum $(\Sigma, \gamma,\gamma',\gamma'',z)$ such that $\gamma,\gamma',\gamma''$ are obtained from $\beta,\beta',\beta''$ by handleslides and isotopy and both the Heegaard data $(\Sigma, \beta,\beta',\beta'',\gamma,\gamma',\gamma'',z)$ and $(\Sigma, \gamma,\gamma',\gamma'', \tbeta,\tbeta', \tbeta'', z)$ are weakly admissible. We choose an equivariant admissible complex structure and modify the perturbations $\beta,\beta',\beta''$ so that the result of Proposition \ref{same} holds, as in subsection \ref{choices}. Then by Proposition \ref{iso} we have the following diagrams, which should be viewed as chain maps $\Psi_1=\mu^{Tw}_2(\cdot, \Theta): \CFh(\alpha, \twibeta)\to \CFh(\alpha, \twigamma)$ and $\Psi_2=\mu^{Tw}_2(\cdot,\Theta):\CFh(\alpha,\twigamma)\to \CFh(\alpha,\twitbeta)$ respectively, with the indicated vertical components:
	\begin{equation}\label{compare1}
	\begin{tikzcd}[column sep=huge, row sep=large]
		\CFh(\alpha, \beta) \arrow[r,"{\mu_2(\cdot, \Theta_{\beta\beta'})}"]\arrow[d,"{\mu_2(\cdot, \Theta_{\beta\gamma})}"]\arrow[rd]\arrow[rrd]\arrow[rr,bend left]&\CFh(\alpha,\beta') \arrow[r,"{\mu_2(\cdot,\Theta_{\beta'\beta''})}"]\arrow[rd]\arrow[d,"{\mu_2(\cdot, \Theta_{\beta'\gamma'})}",crossing over]&\CFh(\alpha,\beta'')\arrow[d,"{\mu_2(\cdot,\Theta_{\beta''\gamma''})}"]\\
		\CFh(\alpha,\gamma)\arrow[r,"{\mu_2(\cdot,\Theta_{\gamma\gamma'})}"]\arrow[rr, bend right]&\CFh(\alpha,\gamma')\arrow[r,"{\mu_2(\cdot,\Theta_{\gamma'\gamma''})}"]&\CFh(\alpha,\gamma'')\\
	\end{tikzcd}
\end{equation}
	\begin{equation}\label{compare2}
	\begin{tikzcd}[column sep=huge, row sep=large]
		\CFh(\alpha, \gamma) \arrow[r,"{\mu_2(\cdot, \Theta_{\gamma\gamma'})}"]\arrow[d,"{\mu_2(\cdot, \Theta_{\gamma\tbeta})}"]\arrow[rd]\arrow[rrd]\arrow[rr,bend left]&\CFh(\alpha,\gamma') \arrow[r,"{\mu_2(\cdot,\Theta_{\gamma'\gamma''})}"]\arrow[rd]\arrow[d,"{\mu_2(\cdot, \Theta_{\gamma'\tbeta'})}",crossing over]&\CFh(\alpha,\gamma'')\arrow[d,"{\mu_2(\cdot,\Theta_{\gamma''\tbeta''})}"]\\
		\CFh(\alpha,\tbeta)\arrow[r,"{\mu_2(\cdot,\Theta_{\tbeta\tbeta'})}"]\arrow[rr, bend right]&\CFh(\alpha,\tbeta')\arrow[r,"{\mu_2(\cdot,\Theta_{\tbeta'\tbeta''})}"]&\CFh(\alpha,\tbeta'')\\
	\end{tikzcd}
\end{equation}

Now since the vertical chains maps are all quasi-isomorphisms, by the filtration induced form the columns of diagrams (\ref{compare1}) and (\ref{compare2}), the maps on homology $\Psi_1: \HFh(\alpha, \twibeta)\to \HFh(\alpha, \twigamma)$ and $\Psi_2:\HFh(\alpha,\twigamma)\to \HFh(\alpha,\twitbeta)$ are also quasi-isomorphisms. From the classical context of exact triangle(for example \cite[Theorem 4.7]{OS05}) we know $\HFh(\alpha,\twitbeta)$ is trivial, thus $\HFh(\alpha, \twibeta)$ is also trivial, which means that the map

\begin{align}\label{map}
	&\Cone(\CFh(\alpha, \beta)\xrightarrow{\mu_2(\cdot,\Theta_{\beta\beta'})}\CFh(\alpha,\beta'))\to \CF(\alpha,\beta'')\\
	&(x,y)\mapsto \mu_2(x,\theta)+\mu_3(x,\Theta_{\beta\beta'},\Theta_{\beta'\beta''})+\mu_2(y,\Theta_{\beta'\beta''})
\end{align}
is a quasi isomorphism. 

To conclude the proof, we can apply Proposition \ref{same}, which says that all the maps in the upper horizontal line of diagram (\ref{compare2}) is $\ZZ/2\ZZ$-equivariant by identifying perturbed chain complexes $\CFh(\alpha,\beta)$, $\CFh(\alpha,\beta')$, $\CFh(\alpha,\beta'')$ with the unperturbed chain complexes $\CFh(\pralpha,\prbeta)$, $\CFh(\pralpha,\prbeta')$, $\CFh(\pralpha,\prbeta'')$,  where the latter are used to define the equivariant Heegaard Floer homology $\HFheq$. This is guaranteed by Proposition \ref{same}, and the choices of the almost complex structure, and the choices we made for the generators $\Theta_{\beta\beta'}$, $\Theta_{\beta'\beta''}$, which are close to invariant generators in the corresponding pre-involutive diagrams(notice that $\theta$ is actually $0$ since there is no generator in the perturbed diagram $(\Sigma,\beta,\beta',\beta'',z)$ of grading $1$). Thus, the maps in equation (\ref{map}) are also equivariant in the above sense. Theorem \ref{main} is proved by identifying $\CFh(\alpha,\beta)$, $\CFh(\alpha,\beta')$, $\CFh(\alpha,\beta'')$ back with $\CFh(\pralpha,\prbeta)$, $\CFh(\pralpha,\prbeta')$, $\CFh(\pralpha,\prbeta'')$ by Proposition \ref{same}.
\qed

\begin{remark}
	For readers not familiar with the language of $A_\infty$ category, we state Proposition \ref{iso} and the proof of the main theorem in the classical setting.
	The sentence ``Then we can complete $\Theta_{\gamma\tgamma}$, $\Theta_{\gamma'\tgamma'}$ and $\Theta_{\gamma''\tgamma''}$ into a cycle between twisted complexes by choosing the diagonal chains appropriately'' in Proposition \ref{iso} means there exists $\Theta_{(0,1)}\in \CFh(\gamma,\tgamma')$, $\Theta_{(1,2)}\in\CFh(\gamma',\tgamma'')$ and 
	$\Theta_{(0,2)}\in \CFh(\gamma,\tgamma'')$ such that the following equations hold:
		\begin{align}
		&\mu_1(\Theta_{(0,1)})=\mu_2(\Theta_{\gamma\gamma'},\Theta_{\gamma'\tgamma'})+\mu_2(\Theta_{\gamma\tgamma}, \Theta_{\tgamma\tgamma'})\label{Theta01}\\
		&\mu_1(\Theta_{(1,2)})=\mu_2(\Theta_{\gamma'\gamma''},\Theta_{\gamma''\tgamma''})+\mu_2(\Theta_{\gamma'\tgamma'}, \Theta_{\tgamma'\tgamma''})\label{Theta12}\\
		&\mu_1(\Theta_{(0,2)})=\mu_2(\theta,\Theta_{\gamma''\tgamma''})+\mu_2(\Theta_{\gamma\tgamma},\ttheta)+\mu_2(\Theta_{\gamma\gamma'},\Theta_{(1,2)})+\mu_2(\Theta_{(0,1)},\Theta_{\tgamma'\tgamma''})+\mu_3(\Theta_{\gamma\gamma'},\Theta_{\gamma'\gamma''},\Theta_{\gamma''\tgamma''})\label{Theta02}\\
		&+\mu_3(\Theta_{\gamma\tgamma},\Theta_{\tgamma\tgamma'},\Theta_{\tgamma'\tgamma''})+\mu_3(\Theta_{\gamma\gamma'},\Theta_{\gamma'\tgamma'},\Theta_{\tgamma'\tgamma''})\nonumber
		\end{align}

	Now we want to verify that 

		\begin{align}\label{verify}
			&\mu_1(\mu_2(\theta,\Theta_{\gamma''\tgamma''})+\mu_2(\Theta_{\gamma\tgamma},\ttheta)+\mu_2(\Theta_{\gamma\gamma'},\Theta_{(1,2)})+\mu_2(\Theta_{(0,1)},\Theta_{\tgamma'\tgamma''})+\mu_3(\Theta_{\gamma\gamma'},\Theta_{\gamma'\gamma''},\Theta_{\gamma''\tgamma''})\\
			&+\mu_3(\Theta_{\gamma\tgamma},\Theta_{\tgamma\tgamma'},\Theta_{\tgamma'\tgamma''})+\mu_3(\Theta_{\gamma\gamma'},\Theta_{\gamma'\tgamma'},\Theta_{\tgamma'\tgamma''}))=0,\nonumber
		\end{align}
which corresponds to the sentence `` which means that $\mu^{Tw}_1(\Theta')=\mu^{Tw}_1(\Theta')_{(0,2)}$ is a cycle in the  grading $1$ part of $\CFh(\gamma,\tgamma'')$'' is the proof of Proposition \ref{iso}.
	Notice that all the horizontal and vertical chains of length $1$ in Figure (\ref{cycle}) are cycles, thus:
\begin{equation}
		\begin{aligned}\label{part1}
		     &\mu_1(\mu_2(\Theta_{\gamma\gamma'},\Theta_{(1,2)}))=\mu_2(\Theta_{\gamma\gamma'},\mu_1(\Theta_{(1,2)}))\\
		     &=\mu_2(\Theta_{\gamma\gamma'},\mu_2(\Theta_{\gamma'\gamma''},\Theta_{\gamma''\tgamma''}))+\mu_2(\Theta_{\gamma\gamma'},\mu_2(\Theta_{\gamma'\tgamma'}, \Theta_{\tgamma'\tgamma''}))\\
		     &\mu_1(\mu_2(\Theta_{(0,1)},\Theta_{\tgamma'\tgamma''}))=\mu_2(\mu_1(\Theta_{(0,1)}),\Theta_{\tgamma'\tgamma''})\\
		     &=\mu_2(\mu_2(\Theta_{\gamma\gamma'},\Theta_{\gamma'\tgamma'}),\Theta_{\tgamma'\tgamma''})+\mu_2(\mu_2(\Theta_{\gamma\tgamma}, \Theta_{\tgamma\tgamma'}),\Theta_{\tgamma'\tgamma''}))
		\end{aligned}
	\end{equation}
	Also:
\begin{equation}
	   \begin{aligned}\label{part2}
	   	&\mu_1(\mu_2(\theta,\Theta_{\gamma''\tgamma''}))=\mu_2(\mu_1(\theta),\Theta_{\gamma''\tgamma''})\\
	   	&=\mu_2(\mu_2(\Theta_{\gamma\gamma'},\Theta_{\gamma'\gamma''}),\Theta_{\gamma''\tgamma''})\\
	   	&\mu_1(\mu_2(\Theta_{\gamma\tgamma},\ttheta))=\mu_2(\Theta_{\gamma\tgamma},\mu_1(\ttheta))\\
	   	&=\mu_2(\Theta_{\gamma\tgamma},\mu_2(\Theta_{\tgamma\tgamma'},\Theta_{\tgamma'\tgamma''}))
	   \end{aligned}
\end{equation}
	And:
\begin{equation}
	\begin{aligned}\label{part3}
		&\mu_1(\mu_3(\Theta_{\gamma\gamma'},\Theta_{\gamma'\gamma''},\Theta_{\gamma''\tgamma''}))=\mu_2(\mu_2(\Theta_{\gamma\gamma'},\Theta_{\gamma'\gamma''}),\Theta_{\gamma''\tgamma''})+\mu_2(\Theta_{\gamma\gamma'},\mu_2(\Theta_{\gamma'\gamma''},\Theta_{\gamma''\tgamma''}))\\
		&\mu_1(\mu_3(\Theta_{\gamma\tgamma},\Theta_{\tgamma\tgamma'},\Theta_{\tgamma'\tgamma''}))=\mu_2(\mu_2(\Theta_{\gamma\tgamma},\Theta_{\tgamma\tgamma'}),\Theta_{\tgamma'\tgamma''})+\mu_2(\Theta_{\gamma\tgamma},\mu_2(\Theta_{\tgamma\tgamma'},\Theta_{\tgamma'\tgamma''}))\\
		&\mu_1(\mu_3(\Theta_{\gamma\gamma'},\Theta_{\gamma'\tgamma'},\Theta_{\tgamma'\tgamma''}))=\mu_2(\mu_2(\Theta_{\gamma\gamma'},\Theta_{\gamma'\tgamma'}),\Theta_{\tgamma'\tgamma''})+\mu_2(\Theta_{\gamma\gamma'},\mu_2(\Theta_{\gamma'\tgamma'},\Theta_{\tgamma'\tgamma''}))
	\end{aligned}
\end{equation}
\end{remark}
Now adding up equations (\ref{part1}), (\ref{part2}), (\ref{part3}) we get equation (\ref{verify}).

Then we turn to the proof of Theorem \ref{main}. In the remaining of this section we will not distinguish between $\theta$ and $\ttheta$ and write both of them as $\theta$. The core part of the proof of Theorem \ref{main} is to construct quasi-isomorphisms $\Psi_1: \CFh(\alpha, \twibeta)\to \CFh(\alpha, \twigamma)$ and $\Psi_2:\CFh(\alpha,\twigamma)\to \CFh(\alpha,\twitbeta)$. We only do this for $\Psi_1$ since the case of $\Psi_2$ is similar. Recall that the underlying group of the chain complex $\CFh(\alpha, \twibeta)$ is $\CFh(\alpha,\beta)\oplus\CFh(\alpha,\beta')\oplus \CFh(\alpha,\beta'')$, with the chain complex structure given as the mapping cone of $\CFh(\alpha,\beta)\oplus\CFh(\alpha,\beta')\to \CFh(\alpha,\beta'')$ as in equation (\ref{map}), where $\CFh(\alpha,\beta)\oplus\CFh(\alpha,\beta')$ has the chain complex structure given by the chain map $\mu_2(\cdot,\Theta_{\beta\beta'}):\CFh(\alpha,\beta)\to\CFh(\alpha,\beta')$. Similar statement also holds for $\CFh(\alpha, \twigamma)$. 

Next we spell out the explicit chain map $\Psi_1$ using the cycle $\Theta$ we constructed in Proposition \ref{iso}. Like the notation for $\Theta$, we write the $(i,j)$ component of $\Psi_1$ as $\Psi_{1(i,j)}$($0\le i\le j\le2$).

	\begin{align}
		&\Psi_{1(0,0)}=\mu_2(\cdot,\Theta_{\beta\gamma}), \Psi_{1(1,1)}=\mu_2(\cdot,\Theta_{\beta'\gamma'}), \Psi_{1(2,2)}=\mu_2(\cdot,\Theta_{\beta''\gamma''})\\
		&\Psi_{1(0,1)}=\mu_2(\cdot,\Theta_{(0,1)})+\mu_3(\cdot,\Theta_{\beta\beta'},\Theta_{\beta'\gamma'})+\mu_3(\cdot,\Theta_{\beta\gamma},\Theta_{\gamma\gamma'})\\
		&\Psi_{1(1,2)}=\mu_2(\cdot,\Theta_{(1,2)})+\mu_3(\cdot,\Theta_{\beta'\beta''},\Theta_{\beta''\gamma''})+\mu_3(\cdot,\Theta_{\beta'\gamma'},\Theta_{\gamma'\gamma''})\\
		&\Psi_{1(0,2)}=\mu_2(\cdot,\Theta_{(0,2)})+\mu_3(\cdot,\theta,\Theta_{\beta''\gamma''})+\mu_3(\cdot,\Theta_{\beta\gamma},\theta)\\
		&+\mu_3(\cdot,\Theta_{\beta\beta'},\Theta_{(1,2)})+\mu_3(\cdot,\Theta_{(0,1)},\Theta_{\gamma'\gamma''})+\mu_4(\cdot,\Theta_{\beta\beta'},\Theta_{\beta'\beta''},\Theta_{\beta''\gamma''})\nonumber\\
		&+\mu_4(\cdot,\Theta_{\beta\beta'},\Theta_{\beta'\gamma'},\Theta_{\gamma'\gamma''})+\mu_4(\cdot,\Theta_{\beta\gamma},\Theta_{\gamma\gamma'},\Theta_{\gamma'\gamma''}).\nonumber
	\end{align}

Next we verify that $\Psi_{1}$ is a chain map.

On $(0,0), (1,1)$ and $(2,2)$ components we need to verify that 
\begin{align}
	&\mu_2(\mu_1(\cdot),\Theta_{\beta\gamma})=\mu_1(\mu_2(\cdot,\Theta_{\beta\gamma}))\nonumber\\
	&\mu_2(\mu_1(\cdot),\Theta_{\beta'\gamma'})=\mu_1(\mu_2(\cdot,\Theta_{\beta'\gamma'}))\nonumber\\
	&\mu_2(\mu_1(\cdot),\Theta_{\beta''\gamma''})=\mu_1(\mu_2(\cdot,\Theta_{\beta''\gamma''}))\nonumber
	\end{align}
	respectively, which follow from that  $\Theta_{\beta\gamma}, \Theta_{\beta'\gamma'}$ and $\Theta_{\beta''\gamma''}$ are cycles, respectively.

On $(0,1)$ and $(1,2)$ components we need to verify
	\begin{align}
			&\mu_2(\mu_1(\cdot),\Theta_{(0,1)})+\mu_3(\mu_1(\cdot),\Theta_{\beta\beta'},\Theta_{\beta'\gamma'})+\mu_3(\mu_1(\cdot),\Theta_{\beta\gamma},\Theta_{\gamma\gamma'})+\mu_2(\mu_2(\cdot,\Theta_{\beta\beta'}),\Theta_{\beta'\gamma'})\label{01}\\
		&=\mu_1(\mu_2(\cdot,\Theta_{(0,1)}))+\mu_1(\mu_3(\cdot,\Theta_{\beta\beta'},\Theta_{\beta''\gamma'}))+\mu_1(\mu_3(\cdot,\Theta_{\beta\gamma},\Theta_{\gamma\gamma'}))+\mu_2(\mu_2(\cdot,\Theta_{\beta\gamma}),\Theta_{\gamma\gamma'})\nonumber\\
	&\mu_2(\mu_1(\cdot),\Theta_{(1,2)})+\mu_3(\mu_1(\cdot),\Theta_{\beta'\beta''},\Theta_{\beta''\gamma''})+\mu_3(\mu_1(\cdot),\Theta_{\beta'\gamma'},\Theta_{\gamma'\gamma''})+\mu_2(\mu_2(\cdot,\Theta_{\beta'\beta''}),\Theta_{\beta''\gamma''})\label{12}\\
	&=\mu_1(\mu_2(\cdot,\Theta_{(1,2)}))+\mu_1(\mu_3(\cdot,\Theta_{\beta'\beta''},\Theta_{\beta''\gamma''}))+\mu_1(\mu_3(\cdot,\Theta_{\beta'\gamma'},\Theta_{\gamma'\gamma''}))+\mu_2(\mu_2(\cdot,\Theta_{\beta'\gamma'}),\Theta_{\gamma'\gamma''})\nonumber
	\end{align}
where the differences of the two sides of equation (\ref{01}) and (\ref{12}) are the $(0,1)$ and $(1,2)$ components of $\Psi_1\underline{\partial}(\cdot)+\underline{\partial}\Psi_1(\cdot)$. Equation (\ref{01}) follows from equation (\ref{Theta01}) and $A_\infty$ relations, and equation (\ref{12}) follows equation (\ref{Theta12}) and $A_\infty$ relations.

On the $(0,2)$ component we need to verify that:
	\begin{align}
		&\mu_2(\mu_1(\cdot),\Theta_{(0,2)})+\mu_3(\mu_1(\cdot),\Theta_{\beta\gamma},\theta)+\mu_3(\mu_1(\cdot),\theta,\Theta_{\beta''\gamma''})\label{02}\\
		&+\mu_3(\mu_1(\cdot),\Theta_{\beta\beta'},\Theta_{(1,2)})+\mu_3(\mu_1(\cdot),\Theta_{(0,1)},\Theta_{\gamma'\gamma''})\nonumber\\
		&+\mu_4(\mu_1(\cdot),\Theta_{\beta\beta'},\Theta_{\beta'\beta''},\Theta_{\beta''\gamma''})+\mu_4(\mu_1(\cdot),\Theta_{\beta\beta'},\Theta_{\beta'\gamma'},\Theta_{\gamma'\gamma''})+\mu_4(\mu_1(\cdot),\Theta_{\beta\gamma},\Theta_{\gamma\gamma'},\Theta_{\gamma'\gamma''})\nonumber\\
		&+\mu_2(\mu_2(\cdot,\Theta_{\beta\beta'}),\Theta_{(1,2)})+\mu_3(\mu_2(\cdot,\Theta_{\beta\beta'}),\Theta_{\beta'\beta''},\Theta_{\beta''\gamma''})+\mu_3(\mu_2(\cdot,\Theta_{\beta\beta'}),\Theta_{\beta'\gamma'},\Theta_{\gamma',\gamma''})\nonumber\\
		&+\mu_2(\mu_2(\cdot,\theta),\Theta_{\beta''\gamma''})+\mu_2(\mu_3(\cdot,\Theta_{\beta\beta'},\Theta_{\beta'\beta''}),\Theta_{\beta''\gamma''})\nonumber\\
		&=\mu_1(\mu_2(\cdot,\Theta_{(0,2)})+\mu_1(\mu_3(\cdot,\theta,\Theta_{\beta''\gamma''}))+\mu_1(\mu_3(\cdot,\Theta_{\beta\gamma},\theta))\nonumber\\
		&+\mu_1(\mu_3(\cdot,\Theta_{\beta\beta'},\Theta_{(1,2)})))+\mu_1(\mu_3(\cdot,\Theta_{(0,1)},\Theta_{\gamma'\gamma''}))\nonumber\\
		&+\mu_1(\mu_4(\cdot,\Theta_{\beta\beta'},\Theta_{\beta'\beta''},\Theta_{\beta''\gamma''}))+\mu_1(\mu_4(\cdot,\Theta_{\beta\beta'},\Theta_{\beta'\gamma'},\Theta_{\gamma'\gamma''}))+\mu_1(\mu_4(\cdot,\Theta_{\beta\gamma},\Theta_{\gamma\gamma'},\Theta_{\gamma'\gamma''}))\nonumber\\
		&+\mu_2(\mu_2(\cdot,\Theta_{(0,1)}),\Theta_{\gamma'\gamma''})+\mu_2(\mu_3(\cdot,\Theta_{\beta\beta'},\Theta_{\beta'\gamma'}),\Theta_{\gamma'\gamma''})+\mu_2(\mu_3(\cdot,\Theta_{\beta\gamma},\Theta_{\gamma\gamma'}),\Theta_{\gamma'\gamma''})\nonumber\\
		&+\mu_2(\mu_2(\cdot,\Theta_{\beta\gamma}),\theta)+\mu_3(\mu_2(\cdot,\Theta_{\beta\gamma}),\Theta_{\gamma\gamma'},\Theta_{\gamma'\gamma''})\nonumber.
		\end{align}
		
	   The difference of the two sides of equation (\ref{02}) is the $(0,2)$ component of $\Psi_1\underline{\partial}(\cdot)+\underline{\partial}\Psi_1(\cdot)$. Equation (\ref{02}) follows from equation (\ref{Theta02}) as well as $A_\infty$ relations.
		\section[spectral]{Proof of The Spectral Sequence}
The proof of the spectral sequence Theorem \ref{spectral} is exactly modeled on Ozsv\'{a}th and Szab\'{o}'s paper \cite{OS05}, but here we use the language of $A_\infty$ category to simplify the algebra as in \cite{Nah25a}. 
\subsection{The bridge diagram for the spectral sequence}
Given a link $L$, we consider a generic projection of $L$, with crossing points labeled by $1$ through $k$. We choose some disjoint balls $D_1, \hdots, D_k$ containing the crossing points $1,\hdots k$ respectively. For each mult-index $I\in\Index$, we form a resolution link $L_I$ of $L$ as well as a division of $L_I$ into $A$, $B$ arcs satisfying the following conditions.
\begin{enumerate}
	\item[A1)]For all $I\in\Index$, $L_I$ as well as the $A$, $B$ arc division of $L_I$ are the same outside $D_1\cup D_2\hdots \cup D_k$.
	\item[A2)]For each $i=1,\hdots, k$, we identify $D_i$ with a disk containing the right of Figure \ref{skeindiagram} in an orientation preserving way. Under this identification $(D_i,L_I\cap D_i)$ as well as its $A$, $B$ arcs decomposition is determined by the $i$-th entry $I_i$ of $I$: If $I_i=\infty$, then $(D_i,L_I\cap D_i)$ and its $A$, $B$ arcs decomposition is identified with the $A, B'$ arcs in the right of Figure \ref{skeindiagram}; If $I_i=0$, then $(D_i,L_I\cap D_i)$ and its $A$, $B$ arcs decomposition is identified with the $A$, $B''$ arcs in the right of Figure \ref{skeindiagram}; If $I_i=1$, then $(D_i,L_I\cap D_i)$ and its $A$, $B$ arcs decomposition is identified with the $A, B'$ arcs in the right of Figure \ref{skeindiagram}. The reader can do isotopy to restore the usual figure of $0$, $1$, $\infty$ resolutions (as in Figure \ref{skein}) from the above arrangement.
\end{enumerate}    

Our next step is to do perturbations on the $B$ arcs so they permute in a specific way around the endpoints of the arcs.
We give the set $\Index$ a lexicographical order $<$ with the understanding $0<1<\infty$, and we give $\Index$ a total order $\prec$ which respects the lexicographical ordering(which can be constructed by picking a maximal element each time and induct).

We construct the perturbations as follows. Given a $B$ arc $B$(there may be many resolutions $L_I$ containing $B$), denote the two endpoints of $B$ by $b_1, b_2$. For each mult-index $I$ such that $L_I$ contains $B$, we wiggle $B$ while fixing $b_1$ and $b_2$ such that the interior of the resulting $B$ arc (denoted by $B'_I$) intersects the interior of $B$ in only one point, and the intersection is transverse; At the endpoints $b_1$ and $b_2$, the tangent vector of $B'_I$ is rotated counterclockwisely in a small angle(The reader may find it convenient to compare with the left of Figure \ref{skeindiagram}). Moreover, if $I\prec I'$ and both $L_I$ and $L_{I'}$ contain $B$, then the rotation angle of $B'_{I'}$ is larger than $B'_{I}$. We make the rotation angles small compare to the angle formed by different $A$, $B$ arcs at their common endpoints, so the wigglings of different $B$ arcs do not interfere with each other.

Now we take the branched double cover of the bridge diagram defined above as in section \ref{1.1}, and picking an endpoint $z$ of a $B$-arc outside the crossing regions, and denote the Heegaard datum by $(\Sigma,\balpha,\{\bbeta_I\}_{I\in\Index},z)$. Next we choose an admissible collection of equivaraint almost complex structures $\{J_n\}$ achieving transversality for the Heegaard datum $(\Sigma,\balpha,\{\bbeta_I\}_{I\in\Index},z)$ (Notice that the total order on $\{\bbeta_I\}_{I\in\Index}$ is given by $\prec$). The transversality can be achieved by equivariant almost complex structures thanks to the perturbations we constructed in the previous paragraph and Proposition \ref{transversality}. For each mult-indices $I<I'$, we also choose a chain $\Theta_{I\le I'}$ fixed under the involution which represents the generator of $\HFh(\beta_I,\beta_{I'})$ in grading $0$(like we did in subsection \ref{choices}). 

We will construct a twisted complex structure on $\underline{\{\beta_I\}_{I\in \Index}}$, that is elements $\theta_{I<I'}\in\CFh(\beta_I,\beta_{I'})$ for each mult-indices $I<I'$ such that
\begin{equation}\label{twist}
	\sum_{I=I_0<I_1<I_2\hdots<I_m=I'}\mu_m(\theta_{I_0<I_1}, \theta_{I_1<I_2}, \hdots, \theta_{I_{m-1}<I_m})=0,
\end{equation}
for each $I<I'$, where the summation is over all sequences $I=I_0<I_1<I_2\hdots<I_m=I'$. 

\

We introduce a little notations. We define a metric $d_0$ on $\{0,1,\infty\}$ such that $d(0,1)=d(1,\infty)=1$, $d(0,\infty)=2$. For indices $I\le I'$, we define their distance $d(I,I')=\sum_{i=1}^kd_0(I_i,I'_i)$.
We construct $F_{I\le L'}$ by induction on $d(I,I')$. 

If $d(I,I')=1$, define $\theta_{I<I'}=\Theta_{I<I'}\in\CFh(\beta_I,\beta_{I'})$.

\begin{proposition}\label{make}
	We can make $\theta_{I<I'}=0$ for all $I<I'$ with $d(I,I')\ge2$, so that equation (\ref{twist}) holds. 
\end{proposition}
{\bf Proof:}
The key point is that the chain complexes $\CFh(\beta_I,\beta_{I'})$ has no non-zero element with degree greater than $0$(by inspecting Figure \ref{skeindiagram}). When  $d(I,I')=1$, equation (\ref{twist}) follows from that $\theta_{I<I'}=\Theta_{I<I'}$ is a cycle. When $d(I,I')=2$, equation (\ref{twist}) is proved by the same methods as in Lemma \ref{lem} or the second paragraph of the proof Proposition \ref{iso}, depending on $I$ and $I'$ differ on one entry or two. When $d(I,I')=3$, the left side of equation (\ref{twist}) is zero because if $m\ge3$ then the outcome of $\mu_m$ lies in degree $m-2>0$; If $m=2$, then one of the $\theta_{I<I'}$ must be zero. If $m=1$ $\mu_1(\theta_{I<I'})=0$ since $\theta_{I<I'}=0$.\qed

Now we can form the big chain complex
\[\CFh(\alpha, \underline{\{\beta_I\}_{I\in \Index}})=\bigoplus_{I\in\Index}\CFh(\alpha, \beta_I),\]
and whose differential $D$, is defined by the sum of 
\begin{align}
	&F_{I\le I'}:\CFh(\alpha,\beta_I)\to\CFh(\alpha,\beta_{I'})\\\nonumber
	&x\mapsto \sum_{I=I_0<I_1<\hdots<I_m=I'}\mu_m(x,\theta_{I_0<I_1},\hdots,\theta_{I_{m-1}<I_m})
	\end{align}
	
	for all mult-indices $I\le I'$.
	
	\begin{proposition}
		$D$ is a differential on $\CFh(\alpha, \underline{\{\beta_I\}_{I\in \Index}})$.
	\end{proposition}
	{\bf{Proof}:}
	This follows from Proposition \ref{make} and $A_\infty$ relations.\qed
	
	
	\begin{proposition}\label{quasi}
		There is a quasi-isomorphism over $\Field[\ZZ/2\ZZ]$ between $\CFh(\alpha, \underline{\{\beta_I\}_{I\in\{0,1\}^k}})$ and $\CFh(\alpha, \beta_{\infty^k})$.
	\end{proposition}
	
	{\bf{Proof:}}
	 For each $0\le j\le k-1$ we claim that there is a quasi-isomorphism over $\Field[\ZZ/2\ZZ]$ between $\CFh(\alpha, \underline{\{\beta_I\}_{I\in\{0,1\}^j\times\infty^{k-j}}})$ and $\CFh(\alpha, \underline{\{\beta_I\}_{I\in\{0,1\}^{j+1}\times\infty^{k-j-1}}})$. Consider the chain complex $\CFh(\alpha,\underline{\{\beta_I\}_{I\in\{0,1\}^{j}\times\{0,1,\infty\}\times\infty^{k-j-1}}})$. There is a filtration by the lexicographical order on the first $j$ coordinates, and the first page of the spectral sequence on each filtration level is $\CFh(\alpha, \underline{\{\beta_I\}_{I\in\epsilon\times\{0,1,\infty\}\times \infty^{k-j-1}}})$ for some $\epsilon\in\{0,1\}^{j}$, which is acyclic by the same proof as Theorem \ref{main}. This proves that the complex $\CFh(\alpha,\underline{\{\beta_I\}_{I\in\{0,1\}^{j}\times\{0,1,\infty\}\times\infty^{k-j-1}}})$ is acyclic. However the chain complex $\CFh(\alpha,\underline{\{\beta_I\}_{I\in\{0,1\}^{j+1}\times\infty^{k-j-1}}})$ can also be seen as a mapping cone between the complexes $\CFh(\alpha, \underline{\{\beta_I\}_{I\in\{0,1\}^j\times\infty^{k-j}}})$ and $\CFh(\alpha, \underline{\{\beta_I\}_{I\in\{0,1\}^{j+1}\times\infty^{k-j-1}}})$, and the map between them is $\ZZ/2\ZZ$-equivariant by the construction of the $\theta$ elements, thus these two chain complex are quasi-isomorphic over $\Field[\ZZ/2\ZZ]$.\qed

	 \

	 {\bf{Proof of Theorem \ref{spectral}:}}

	 Let $R=\Field[q]/q^N$ for some $N$.
	 Now we shift to cohomology and Proposition \ref{quasi} reads there is a quasi-isomorphism over $\Field[\ZZ/2\ZZ]$ between $\CFh^*(\alpha, \underline{\{\beta_I\}_{I\in\{0,1\}^k}})$ and $\CFh^*(\alpha, \beta_{\infty^k})$, only now for each mult-indices $I\le I'$ there is a dual map $F^*_{I\le I'}:\CFh^*(\alpha, \beta_{I'})\to\CFh^*(\alpha, \beta_{I})$. Consider the corresponding projection resolutions, we have a quasi-isomorphism between $\CFh^*(\alpha, \underline{\{\beta_I\}_{I\in\{0,1\}^k}})\tensor R$ and $\CFh^*(\alpha, \beta_{\infty^k})\tensor R$ by considering the spectral sequence induced by the power of $q$.

	 Now consider the filtration on $\CFh^*(\alpha, \underline{\{\beta_I\}_{I\in\{0,1\}^k}})\tensor R$ given by the lexicographical order coupled with the power of $q$. The $E^1$ page of the spectral sequence is $\bigoplus_{I\in\{0,1\}^k}\HFh^*(\alpha, \beta_I)$. Since each of $L_I$ represents an unlink, $\HFh^*(\alpha, \beta_I)$ is naturally isomorphic to \[\bigotimes_{C\in\pi_0(L_I), z\notin C}V_C,\] where $V_C=V$ is a $2$-dimensional vector space over $\Field$ with one generator in each of the gradings $\pm\frac{1}{2}$. Notice that the $\ZZ/2\ZZ$ action on each of the $\HF^*(\alpha, \beta_I)$ component is trivial, thus the $E^1$ page of the spectral sequence has only differential coming from the dual maps $F^*_{I<I'}:\HFh^*(\alpha, \beta_{I'})\to\HFh^*(\alpha, \beta_{I})$. By the same argument as in \cite[Section 6]{OS05}, we conclude that the second page of the spectral sequence is isomorphic to $\widetilde{Kh}(L)\tensor R$.\qed

		\

		\bibliographystyle{alpha}
		\bibliography{reference}
		
	\end{document}